\documentclass[psamsfonts]{amsart}

\usepackage[margin=2.5cm]{geometry}
\usepackage[all,arc]{xy}
\usepackage{amsfonts}
\usepackage{amssymb}
\usepackage{amsaddr}
\usepackage{enumerate}
\usepackage{eucal}
\usepackage{tikz}
\usepackage{parskip}

\newtheorem{thm}{Theorem}[section]
\newtheorem{cor}[thm]{Corollary}
\newtheorem{prop}[thm]{Proposition}
\newtheorem{lemma}[thm]{Lemma}

\theoremstyle{definition}
\newtheorem{defn}[thm]{Definition}

\newtheorem{exmp}[thm]{Example}

\theoremstyle{remark}
\newtheorem{rem}[thm]{Remark}


\numberwithin{equation}{section}

\DeclareMathOperator{\aut}{Aut}

\DeclareMathOperator{\en}{End}
\DeclareMathOperator{\heit}{ht}

\DeclareMathOperator{\Imm}{Im}
\DeclareMathOperator{\irr}{Irr}

\DeclareMathOperator{\lie}{Lie}

\DeclareMathOperator{\adm}{Adm}
\DeclareMathOperator{\coprin}{CoPr}
\DeclareMathOperator{\prin}{Pr}
\DeclareMathOperator{\coorprin}{(Co)Pr}

\DeclareMathOperator{\res}{Res}

\DeclareMathOperator{\spec}{Spec}
\DeclareMathOperator{\str}{STr}
\DeclareMathOperator{\tr}{Tr}
\DeclareMathOperator{\zhu}{Zhu}

\DeclareMathOperator{\supp}{Supp}

\newcommand{\C}{\mathbb{C}}

\newcommand{\R}{\mathbb{R}}
\newcommand{\Q}{\mathbb{Q}}
\newcommand{\Z}{\mathbb{Z}}

\newcommand{\CC}{\mathcal{C}}

\newcommand{\HH}{\mathcal{H}}

\newcommand{\OO}{\mathcal{O}}

\newcommand{\W}{\mathcal{W}}

\newcommand{\mb}{\mathfrak{b}}
\newcommand{\g}{\mathfrak{g}}
\newcommand{\h}{\mathfrak{h}}

\newcommand{\n}{\mathfrak{n}}

\newcommand{\al}{\alpha}
\newcommand{\ga}{\gamma}

\newcommand{\D}{\Delta}
\newcommand{\eps}{\epsilon}
\newcommand{\la}{\lambda}
\newcommand{\La}{\Lambda}
\newcommand{\om}{\omega}

\newcommand{\ov}{\overline}
\newcommand{\vac}{{\left|0\right>}}

\newcommand{\twobytwo}[4]
{\left(\begin{smallmatrix} #1 & #2 \\ #3 & #4 \end{smallmatrix}\right)}

\newcommand{\cliff}{\mathcal{C}\ell}

\DeclareMathOperator{\height}{ht}

\newcounter{ForTheListInLemma34}

\title[Modularity and Affine $W$-Algebras]{Modularity of Relatively Rational Vertex Algebras and Fusion Rules of Principal Affine $W$-Algebras}
\author[*]{Tomoyuki Arakawa}
\author[**]{Jethro van Ekeren}

\begin{document}

\begin{center}
{\LARGE \bf Modularity of Relatively Rational Vertex Algebras and Fusion Rules of Principal Affine $W$-Algebras} \par \bigskip

\renewcommand*{\thefootnote}{\fnsymbol{footnote}}
{\normalsize
Tomoyuki Arakawa\textsuperscript{1},
Jethro van Ekeren\footnote{email: \texttt{jethrovanekeren@gmail.com}}\textsuperscript{2}
}

\par \bigskip

\textsuperscript{1}{\footnotesize Research Institute for Mathematical Sciences, Kyoto, Japan/MIT, Cambridge, USA}

\par

\textsuperscript{2}{\footnotesize Instituto de Matem\'{a}tica e Estat\'{i}stica (GMA), UFF, Niter\'{o}i RJ, Brazil}

\par \bigskip
\end{center}

\vspace*{10mm}

\noindent
\textbf{Abstract.} We study modularity of the characters of a vertex (super)algebra equipped with a family of conformal structures. Along the way we introduce the notions of rationality and cofiniteness relative to such a family. We apply the results to determine modular transformations of trace functions on admissible modules over affine Kac-Moody algebras and, via BRST reduction, trace functions on minimal series representations of principal affine $W$-algebras.

\vspace*{10mm}

\section{Introduction}

A striking feature of the representation theory of infinite dimensional Lie algebras and vertex algebras is the appearance of modular functions as normalised graded dimensions of integrable modules. The phenomenon of modularity is in turn the source of important technical tools in the representation theory of these algebras.

Let $(V, \om)$ be a conformal vertex algebra of central charge $c$, assumed to be rational and $C_2$-cofinite. In \cite{Z96} Zhu proved modular invariance of the normalised graded dimensions of irreducible positive energy $V$-modules, and more generally of the trace functions
\[
S_M(\tau | u) = q^{-c/24} \sum_{n=0}^\infty \tr_{M_n} u_0 q^{L_0},
\]
of such modules. Here $u \in V$, and $S_M(\tau | u)$ is viewed as a holomorphic function of $\tau \in \HH$ where $q = e^{2\pi i \tau}$ and $\HH$ is the upper half complex plane. In particular
\begin{align}\label{Zhu.S}
S_M(-1/\tau | \tau^{-L_{[0]}} u) = \sum_{M'} S_{M, M'} S_{M'}(\tau | u)
\end{align}
where
\begin{align*}
L_{[0]} = L_0 - \sum_{j=1}^\infty \frac{(-1)^j}{j(j+1)} L_j
\end{align*}
is the neutral Zhu mode (see (\ref{ZhuConformalZeroMode})), and the sum here is over the set of irreducible positive energy $V$-modules.

The $S$-matrix $\{S_{M,M'}\}$ is an important datum associated with $V$. Indeed the celebrated Verlinde formula \cite{Verlinde88} (proved as a theorem of vertex algebras by Huang \cite{Huang08Contemp}) determines the decomposition multiplicities of the fusion product between $V$-modules in terms of the $S$-matrix.

The specialisation of (\ref{Zhu.S}) to $u = \vac$ recovers modularity of normalised graded dimensions of modules. For the purpose of computing the $S$-matrix it is important to allow the insertion of arbitrary $u \in V$. This is because, while the $S_M(\tau | u)$ are known to be linearly independent, their restrictions $S_M(\tau | \vac)$ need not be.

In the first half of this paper we study modularity in the context of a vertex algebra $V$ together with an infinitesimal variation of its conformal structure $\om$.

Let $(V, \om)$ be a conformal vertex algebra, and let $h \in V$ be a current (i.e., a vector of conformal weight $1$). It is well known that modification of $\om$ by the derivative of $h$ defines a new ``shifted'' conformal vector, so that
\[
\om(z) = \om - zTh
\]
defines a family of conformal structures on $V$, indexed by the parameter $z$. (For explanation of technical terms used in this introduction, we refer the reader to Section \ref{section.prelim}.)

The vertex algebra $(V, \om)$ is said to be rational if its category of positive energy modules is semisimple. Since the positive energy condition depends on a choice of conformal structure, so does the condition of rationality. One is thus presented with the possibility of a family of conformal structures $(V, \om(z))$ as above for which $V$ is ``generically rational'', i.e., rational for all $\om(z)$ in some neighbourhood of $\om$, but not necessarily at $\om$ itself. In fact this situation occurs relatively frequently.

Indeed we may speak of the subcategory of the category of positive energy $(V, \om)$-modules which retain the positive energy condition upon deformation of $\om$ to $\om(z)$ for small $z \in \R_{>0}$. We call such modules ``$h$-stable'', and we say that $(V, \om)$ is rational relative to $h$ if its category of $h$-stable positive energy modules is semisimple.

In \cite{Z96} Zhu introduced the commutative (indeed Poisson) algebra $R(V) = V / V_{(-2)}V$ canonically associated with the vertex algebra $V$. He also identified the condition $\dim{R(V)} < \infty$ as crucial for establishing his modularity theorem. A vertex algebra satisfying this condition is said to be ``$C_2$-cofinite'' or ``lisse''. In the relative setting it is appropriate to replace $C_2$-cofiniteness with a weaker condition involving $V$ and its subalgebra $V^0 \subseteq V$ consisting of vectors that commute with $h$. Accordingly we make the following general definition.
\begin{defn}
Let $V$ be a vertex algebra equipped with a decomposition $V = V^0 \oplus V^+$ as the direct sum of a vertex subalgebra $V^0$ and a $V^0$-module $V^+$. The quotient
\[
R^{\text{rel}}(V) = \frac{V}{V^0_{(-2)}V^0 + V_{(-1)}V^+}
\]
is a commutative algebra. We say that $V$ is \emph{cofinite} relative to the decomposition $V = V^0 \oplus V^+$ if
\[
\dim{R^{\text{rel}}(V)} < \infty.
\]
\end{defn}
We remark that relative cofiniteness is implied by $C_2$-cofiniteness either of $V$ itself or of $V^0$, but the converse is not true. The applications that most interest us involve relatively cofinite but \emph{non} $C_2$-cofinite vertex algebras.

Let $(V, \om)$ and $h$ be as above, and suppose that $V$ decomposes under the action of $h_0$ into the sum $V^0 \oplus V^+$ of the zero eigenspace $V^0$ and a complementary invariant subspace $V^+$. We say that $V$ is cofinite relative to $h$ if it is cofinite relative to the decomposition $V^0 \oplus V^+$.

The neutral mode of $\om(z)$ is $L_0(h) = L_0 + zh_0$. Hence trace functions on $(V, \om(z))$-modules are naturally functions of $z$ alongside $\tau$ and $u \in V$. The following theorem summarises the main results of Section \ref{main.computation}.
{
\begin{thm}\label{main.modularity.intro}
Let $(V, \om)$ be a conformal vertex \textup{(}super\textup{)}algebra graded by integer conformal weights. Let $h \in V$ be a current satisfying the OPE relations
\[
[h_\la h] = 2\la \vac \quad \text{and} \quad [L_\la h] = (T+\la)h + p \tfrac{\la^2}{2} \vac,
\]
where $p$ is some constant, and such that $h_0$ acts semisimply on $V$. Assume $(V, \om)$ to be rational relative to $h$ and cofinite relative to $h$, and write $\mathcal{X}$ for the set of irreducible $h$-stable positive energy $V$-modules. For $u \in V$ and $M \in \mathcal{X}$ we consider the supertrace function
\[
F_{M}(\tau, z | u) = \str_M u_0 e^{2\pi i z(h_0 - p/2)} q^{L_0 - c/24}.
\]
There exists $\varepsilon > 0$ such that for each $M \in \mathcal{X}$ and all $u \in V$ the supertrace function $F_M(\tau, z | u)$ converges absolutely uniformly on compact subsets of the domain
\[
\{(\tau, z) \in \HH \times \C | 0 < \Imm(z) < \varepsilon \Imm(\tau)\}.
\]
We consider the following action of $SL_2(\Z)$ on functions $F$ of $\tau$ and $z$ and linear in $u \in V$:
\begin{align}\label{S-action-general}
[F \cdot A](\tau, z | u) = \exp\left[ -2\pi i \frac{c z^2}{c\tau+d} \right] F\left( \frac{a\tau+b}{c\tau+d}, \frac{z}{c\tau+d} \bigg| (c\tau+d)^{-L_{[0]}} \exp{\left[ -\frac{c z}{c\tau+d} I(h) \right]} u \right),
\end{align}
where
\begin{align}\label{def.ih}
I(h) = \sum_{j=1}^\infty \frac{(-1)^j}{-j} h_j.
\end{align}
Suppose \textup{(}1\textup{)} that the set of functions $F_M(\tau, z | \vac)$ as $M$ runs over $\mathcal{X}$ is modular invariant, i.e., that there exists a representation $\upsilon$ of $SL_2(\Z)$ for which (\ref{Fmodular.repeat}) holds for $u = \vac$, and \textup{(}2\textup{)} that for each $\alpha \in (0, \varepsilon) \subset \R$ and $\beta \in \R$ the set of functions $F_M(\tau, \alpha\tau + \beta | \vac)$, as $M$ runs over $\mathcal{X}$, is linearly independent. Then for all $u \in V$ the relation
\begin{align}\label{Fmodular.repeat.intro}
[F_M \cdot A](\tau, z | u) = \sum_{M' \in \mathcal{X}} \rho_{M, M'}(A) F_{M'}(\tau, z | u)
\end{align}
is satisfied in the intersection of the domains of convergence of the two sides.
\end{thm}
}
We make some remarks on the theorem and its proof. The essential idea is to apply Zhu's modularity theorem to the vertex algebra $(V, \om(z))$. However $\om(z)$ equips $V$ with noninteger conformal weights, and Zhu's theorem does not apply in this case. In \cite{JVE13} it is shown instead that modular transformations map the trace functions $F_M$ to trace functions on particular \emph{twisted} modules. The task becomes to relate trace functions on twisted and untwisted $V$-modules. This is achieved by use of Li's shift operators $\D(u, z)$. The condition of relative cofiniteness is inspired by the work \cite{DLMadmiss}, see also \cite{JVE13}.

The transformation (\ref{S-action-general}) was uncovered in the case of $N=2$ superconformal vertex algebras in {\cite[Theorem 9.13 (b)]{HVE14}}, with $h$ equal to the $U(1)$ current of the $N=2$ algebra. There the functions $F_M$ are shown to be flat sections of the bundle of conformal blocks over the universal elliptic curve, and (\ref{S-action-general}) is derived from the geometry of this bundle. We also note that a result closely related to Theorem \ref{main.modularity.intro} was recently independently obtained in \cite{Krauel.C2.case} in the case of $V$ rational and $C_2$-cofinite (see also \cite{KM15}).

An important class of vertex algebras that are relatively cofinite and generically rational in the sense discussed above is afforded by the simple affine vertex algebras at admissible level.

Let $\ov\g$ be a finite dimensional simple Lie algebra over $\C$, and $\g$ the corresponding affine Kac-Moody algebra. In \cite{KWPNAS} Kac and Wakimoto identified the notion of admissible weight and initiated the study of the characters
\[
\chi_\la(\tau, x) = \tr_{L(\la)} e^{2\pi i x_0} q^{L_0 - c_k/24} \quad \text{(where $x \in \h$, $\tau \in \HH$, and $q=e^{2\pi i \tau}$)}
\]
of the irreducible $\g$-modules of admissible highest weight $\la$.

We recall that $k \in \Q$ is said to be an admissible number for $\ov\g$ if $k\La_0$ is an admissible weight. If $k$ is an admissible number then it is either \emph{principal} or else \emph{coprincipal}. Roughly speaking these cases distinguish whether the integrable root system of $k\La_0$ is equivalent to that of $\g$ or else to that of the Langlands dual ${}^L(\widehat{{}^L\ov\g})$ of the affine algebra associated with ${}^L\ov\g$, respectively (see Section \ref{liealgebras} for precise definitions). We denote by $\prin^k$ (resp.\ $\coprin^k$) the set of principal (resp.\ coprincipal) weights of level $k$.

In \cite{KW89} Kac and Wakimoto showed that if $k \in \Q$ is a principal admissible number for $\ov\g$ and $\la \in \prin^k$ then
\[
\chi_\la\left(\frac{a\tau+b}{c\tau+d}, \frac{x}{c\tau+d}\right) = \exp{\left[ 2\pi i k \frac{c (x, x)}{2(c\tau+d)} \right]} \sum_{\la' \in \prin^k} \rho_{\la, \la'}(A) \chi_{\la'}(\tau, x)
\]
for some representation $\rho$ of $SL_2(\Z)$. They also explicitly computed the $S$-matrix
\[
a({\la, \la'}) = \rho_{\la, \la'}\twobytwo{0}{-1}{1}{0}.
\]
In Section \ref{section.coprin.Smatrix} we extend this result to the coprincipal case, and we compute the $S$-matrix explicitly.

Now let $V^k(\ov\g)$ be the universal affine vertex algebra at admissible level $k$, and $V_k(\ov\g)$ its simple quotient. A smooth ${\g}$-module of level $k$ is naturally a $V^k(\ov\g)$-module. Consider the subcategory of the BGG category $\OO_k$ consisting of modules that descend to $V_k(\ov\g)$-modules. It was 
 conjectured
 in \cite{AdaMil95} and
 proved in {\cite{A12catO}} that this category is semisimple, i.e., that $V_k(\ov\g)$ is \emph{rational in the category $\OO$}. Furthermore if $k$ is principal (resp.\ coprincipal) then the simple objects are precisely the irreducible ${\g}$-modules $L(\la)$ for $\la \in \prin^k$ (resp.\ $\la \in \coprin^k$).

For $\la$ an admissible weight, we introduce the trace function
\begin{align}\label{F.definition.introduction}
\Psi_\la(\tau, x | u) = \tr_{L(\la)} u_0 e^{2\pi i x_0} q^{L_0 - c/24},
\end{align}
of $u \in V_k(\ov\g)$ on $L(\la)$. The Kac-Wakimoto character $\chi_\la$ is recovered from $\Psi_\la$ as the $u=\vac$ specialisation. As an application of Theorem \ref{main.modularity.intro} we prove the following.
\begin{thm}\label{theorem1}
Let $\ov\g$ be a simple Lie algebra and $k \in \Q$ a \textup{(}co\textup{)}principal admissible level for $\ov\g$. For all $\la \in \coorprin^k$ we have
\begin{align*}
{\Psi}_\la\left( \frac{a\tau+b}{c\tau+d}, \frac{x}{c\tau+d} \bigg| (c\tau+d)^{-L_{[0]}} \exp{\left[ -\frac{c}{c\tau+d} {I(x)} \right]} u \right) = \exp\left( 2\pi i k \frac{c (x, x)}{2(c\tau+d)} \right) \sum_{\la' \in \coorprin^k} \rho_{\la, \la'}(A) {\Psi}_{\la'}(\tau, x | u)
\end{align*}
where $\rho$ is a representation of $SL_2(\Z)$. The $S$-matrix $a(\la, \la') = \rho\twobytwo{0}{-1}{1}{0}$ is given in {\cite[Theorem 3.6]{KW89}} \textup{(}see Theorem \ref{KWSmatrix} below\textup{)} if $k$ is principal, and by Theorem \ref{KWSmatrix.coprin} below if $k$ is coprincipal.
\end{thm}

Finally we apply Theorem \ref{theorem1} to solve a problem in the representation theory of affine $W$-algebras.

Recall that from the data of $\ov\g$ and $k$ as above, plus a choice of nilpotent element $f \in \ov\g$, the universal affine $W$-algebra $\W^k(\g, f)$ is defined as the quantized Drinfeld-Sokolov reduction $H^0_f(V^k(\ov\g))$ \cite{FeiginFrenkel}, \cite{KRW03}. We focus on the case of $f$ a principal nilpotent element and $k$ a principal admissible number, and we omit $f$ from the notation.

Let $k$ be a non-degenerate admissible level. It was conjectured 
 in \cite{FKW} and proved in \cite{A.assoc.var} and 
 \cite{A12rational} that the simple quotient $\W_k(\ov\g)$ of $\W^k(\ov\g)$ is a rational and $C_2$-cofinite vertex algebra. Zhu's theorem therefore asserts modularity for $\W_k(\ov\g)$. The $S$-matrix of $\W_k(\ov\g)$ can be deduced from that of $V_k(\ov\g)$ using the Euler-Poincar\'{e} principle. With the $S$-matrix in hand one may use the Verlinde formula to compute the fusion rules of $\W_k(\ov\g)$.

The fusion rules of $\W_k(\ov\g)$, for $\ov\g$ simply laced, were worked out by Frenkel, Kac and Wakimoto in \cite{FKW} by carrying out the calculation outlined above at the level of the characters $\chi_\la$, i.e., at the level of graded dimensions of $\W_k(\ov\g)$-modules. As noted above the graded dimensions are not linearly independent. However Theorem \ref{theorem1} can be used to upgrade the calculation to the level of trace functions of arbitrary $u \in \W_k(\ov\g)$, which are linearly independent, and the result of \cite{FKW} is confirmed.

\emph{Acknowledgements} 
The first author is partially supported by JSPS KAKENHI  Grant Number 17H01086
and 17K18724. The second author was supported by an Alexander von Humboldt Foundation grant and later by CAPES-Brazil. The second author would like to thank Victor G. Kac for several ideas which go back to discussions had with him in 2011. 
Both authors would like to thank the referees 
for their helpful comments.
The work has been presented at conferences ``Lie and Jordan Algebras VI'', Bento Gon\c{c}alves, Brazil, December 2015, ``Qu\'{a}ntum 2016'', C\'{o}rdoba, Argentina, February 2016, and ``Vertex Algebras and Quantum Groups'' Banff, Canada, March 2016. The authors would like to thank the organisers of these conferences.

\emph{Notation} Implicitly tensor products are taken over the ground field $\C$ of complex numbers. The domain of the complex variable $\tau$ is the upper half complex plane, denoted $\HH$, and $q = e^{2\pi i \tau}$. The letter $c$ is used for the central charge, and in the matrix $\twobytwo abcd \in SL_2(\Z)$. We trust that no confusion will arise.

\section{Preliminaries on Vertex Algebras}\label{section.prelim}

\subsection{Vertex Algebras}

For background on vertex algebras we refer the reader to the book \cite{KacVA}. Note that `vertex algebra' implicitly includes the super case.
\begin{defn}\label{VOAdef}
A \emph{vertex algebra} consists of a vector superspace $V$ with a distinguished vacuum vector $\vac \in V$ and a vertex operation, which is an even linear map $V \otimes V \rightarrow V((z))$, written $u \otimes v \mapsto Y(u, z)v = \sum_{n \in \Z} u_{(n)}v z^{-n-1}$, such that the following are satisfied:
\begin{itemize}
\item (Unit axioms) $Y(\vac, z) = 1_V$ and $Y(u, z)\vac \in u + zV[[z]]$ for all $u \in V$.

\item (Borcherds identity)
\begin{align}\label{BorcherdsIdentity}
\sum_{\al \geq 0} \binom{m}{\al} (u_{(n+\al)}v)_{(m+k-\al)}x = \sum_{\al\geq 0} (-1)^\al \binom{n}{\al} \left[ u_{(m+n-\al)} v_{(k+\al)} - (-1)^n p(u, v) v_{(n+k-\al)} u_{(m+\al)} \right] x
\end{align}
for all $u, v, x \in V$, $k, m, n \in \Z$.
\end{itemize}
\end{defn}
The operator $T : u \mapsto u_{(-2)}\vac$ is called the translation operator and it satisfies $Y(Tu, z) = \partial_z Y(u, z)$. The operators $u_{(n)}$ are called \emph{modes}.

A useful special case of Borcherds identity is
\begin{align}\label{commutator.formula}
[u_{(m)}, v_{(n)}] = \sum_{j \in \Z_+} \binom{m}{j} (u_{(j)}v)_{(m+n-j)},
\end{align}
or, in the more compact $\la$-bracket notation,
\[
[u_\la v] = \sum_{j \in \Z_+} \frac{\la^j}{j!} u_{(j)}v.
\]

\begin{defn}
A \emph{conformal structure} on the vertex algebra $V$ is a vector $\om \in V$ such that $Y(\om, z) = L(z) = \sum_{n \in \Z} L_n z^{-n-2}$ furnishes $V$ with an action of the Virasoro algebra, i.e.,
\[
[L_m, L_n] = (m-n)L_{m+n} + \frac{m^3-m}{12}\delta_{m, -n}c
\]
for some constant $c \in \C$. This action is required to satisfy $L_{-1} = T$, and that $L_0$ act semisimply on $V$ with rational eigenvalues, bounded below. The constant $c$ is called the \emph{central charge} of $V$. A \emph{conformal vertex algebra} is a vertex algebra together with a choice of conformal structure.
\end{defn}
After fixing a conformal structure $\om$ on a vertex algebra $V$, we call the $L_0$-eigenvalue of an eigenvector $u \in V$ its \emph{conformal weight}, which we denote $\D(u)$, and we denote by $V_\D$ the subspace of vectors with conformal weight $\D$. The \emph{conformal indexing} of modes (relative to $\om$) is defined by
\[
Y(u, z) = \sum_{n \in \Q} u_n z^{-n-\D(u)}, \quad \text{i.e., $u_{n} = u_{(n+\D(u)-1)}$.}
\]
In terms of the conformal indexing Borcherds identity becomes
\begin{align}\label{ConformalBorcherdsIdentity}
\sum_{\al \geq 0} \binom{m+\D(u)-1}{\al} (u_{(n+\al)}v)_{m+k}x = \sum_{\al\geq 0} (-1)^\al \binom{n}{\al} \left[ u_{m+n-\al} v_{k+\al-n} - (-1)^n p(u, v) v_{k-\al} u_{m+\al} \right] x.
\end{align}

\begin{defn}
Let $V$ be a vertex algebra. A \emph{\textup{(}weak\textup{)} $V$-module} is a vector superspace $M$ together with an even map $Y^M : V \otimes M \rightarrow M((z))$, written $u \otimes x \mapsto Y(u, z)x = \sum_{n \in \Z} u_{(n)}x z^{-n-1}$, such that $Y^M(\vac, z) = 1_M$, and (\ref{BorcherdsIdentity}) holds for all $u, v \in V$, $x \in M$, and for all $m, k, n \in \Z$. Now let $V$ be a conformal vertex algebra. A \emph{positive energy $V$-module} is a weak $V$-module $M, Y^M$ with grading $M = \bigoplus_{\la \in \Q} M_\la$ by finite dimensional $L_0$-eigenspaces $M_\la$, with eigenvalues bounded below.
\end{defn}

An automorphism of the vertex algebra $V$ is $\sigma \in \en{V}$ such that $(\sigma u)_{(n)} = \sigma u_{(n)} \sigma^{-1}$. An automorphism of a conformal vertex algebra is one that fixes $\om$.
\begin{defn}\label{twisted.mod.def}
Let $g$ be an automorphism of the vertex algebra $V$ of finite order $K$, and write $V^\eps$ for its $e^{2\pi i \eps}$-eigenspace (so $\eps$ is defined modulo $\Z$). A \emph{\textup{(}weak\textup{)} $g$-twisted $V$-module} is a vector superspace $M$ together with an even map $Y^M : V \otimes M \rightarrow M((z^{1/K}))$, written $u \otimes x \mapsto Y(u, z)x = \sum_{n \in \eps + \Z} u_{n}x z^{-n-\D(u)}$ for $u \in V^\eps$, such that $Y^M(\vac, z) = 1_M$, and (\ref{ConformalBorcherdsIdentity}) holds for all $u \in V^{\eps}$, $v \in V^{\eps'}$ and $x \in M$, and for all $m \in \eps + \Z$, $k \in \eps' + \Z$ and $n \in \Z$.
\end{defn}
\begin{rem}
For $V$ integer graded our definition coincides with that used in {\cite{DLM98}} and in {\cite{Li95}}. In {\cite{DLM00}} a different convention is used which exchanges the notions of $g$- and $g^{-1}$-twisted modules. The operator $\xi$ to be defined in equation \ref{intertwiner.defn} below is the inverse of $\phi$ used in {\cite[Equation (8.1)]{DLM00}}. Note that in the present setting a vertex algebra $V$ is an $e^{-2\pi i L_0}$-twisted $V$-module. 
\end{rem}

\begin{defn}
A vertex algebra $V$ is said to be \emph{rational} if it has finitely many irreducible positive energy modules, and every positive energy $V$-module decomposes into a direct sum of irreducible positive energy $V$-modules.
\end{defn}

\subsection{Relative Cofiniteness}

We introduce a notion which we call \emph{relative cofiniteness}, generalising the well known \emph{$C_2$-cofiniteness} condition of Zhu \cite{Z96}.
\begin{defn}\label{rel.cofinite}
Let $V = V^0 \oplus V^+$ be a vertex algebra extension of $V^0$ by its module $V^+$. Put
\begin{align*}
C^\text{rel}(V) &= V^0_{(-2)}V^0 + V_{(-1)}V^+ = V_{(-2)}V + V_{(-1)}V^+, \\
R^\text{rel}(V) &= V / C^\text{rel}(V).
\end{align*}
Then we say $V$ is \emph{cofinite} relative to the decomposition $V = V^0 \oplus V^+$ if $\dim R^\text{rel}(V) < \infty$.
\end{defn}
Note that $V^+ \subset C^\text{rel}(V)$, so $R^\text{rel}(V)$ is naturally a quotient of $V^0$. The case $V^+=0$ recovers $C_2$-cofiniteness of $V^0$. On the other hand if $V$ is $C_2$-cofinite then it is cofinite relative to any decomposition $V = V^0 \oplus V^+$.

In this paper we mainly use splittings of the following form: $V^0$ is the fixed point subalgebra of $V$ with respect to a finite order automorphism $g$, and $V^+$ is the sum of the remaining $g$-eigenspaces.
\begin{lemma}\label{tensor.cofinite}
Let $V$ and $W$ be vertex algebras carrying automorphisms of equal order, with $V = V^0 \oplus V^+$ and $W = W^0 \oplus W^+$ the corresponding splittings. If $V$ and $W$ are relatively cofinite, then so is the tensor product $V \otimes W$ with its natural vertex algebra structure and splitting induced by the product automorphism.
\end{lemma}

\begin{proof}
Recall the tensor product vertex algebra structure is $(v \otimes w)_{(n)} = \sum_{j+k = n-1} v_{(j)} \otimes w_{(k)}$. We have
\begin{align*}
C^\text{rel}(V \otimes W)
&= (V \otimes W)_{(-2)}(V \otimes W) + (V \otimes W)_{(-1)}(V \otimes W)^+ \\
& \supset (V_{(-2)}V) \otimes W + V \otimes (W_{(-2)}W) + (V_{(-1)}V^+) \otimes (W_{(-1)}W^0) + (V_{(-1)}V^0) \otimes (W_{(-1)}W^+) \\
&= (V_{(-2)}V) \otimes W + V \otimes (W_{(-2)}W) + (V_{(-1)}V^+) \otimes W + V \otimes (W_{(-1)}W^+) \\
&= C^\text{rel}(V) \otimes W + V \otimes C^\text{rel}(W).
\end{align*}
Hence $R^\text{rel}(V \otimes W)$ is a quotient of $R^\text{rel}(V) \otimes R^\text{rel}(W)$ which is finite dimensional.
\end{proof}

{
\begin{lemma}
The quotient $R^{\text{rel}}(V)$ is a commutative algebra with product $ab = a_{(-1)}b$.
\end{lemma}

\begin{proof}
In {\cite[Section 4.4]{Z96}} Zhu proved the $V^+=0$ case, i.e., that the product $ab = a_{(-1)}b$ is well defined on the quotient $R(V) = V / V_{(-2)}V$.

In the general case $R^\text{rel}(V)$ is the quotient of $R(V)$ by the image of $V_{(-1)}V^+$, so it suffices to show that the latter subspace is an ideal. Let $u, v \in V$ and $w \in V^+$. Putting $m=0$, $k = n = -1$ in (\ref{BorcherdsIdentity}) yields
\[
u_{(-1)}(v_{(-1)}w) \equiv (u_{(-1)}v)_{(-1)}w \pmod{V_{(-2)}V}.
\]
\end{proof}
}

\subsection{Trace Functions and Modular Invariance}\label{tr.fct.subsection}

Let $(V, \om)$ be a conformal vertex algebra, and let $M = \bigoplus_{\la} M_\la$ be an irreducible positive energy $V$-module graded by finite dimensional eigenspaces for $L_0$. We define the \emph{supertrace function} of $u \in V$ on $M$ to be
\[
S_M(\tau | u) = \str_M u_0 q^{L_0-c/24} = \sum_\la q^{\la-c/24} \str_{M_\la} u_0,
\]
wherever the right hand side converges.

More generally let $g_1, g_2$ be commuting finite order automorphisms of $(V, \om)$, and let $M, Y^M$ be an irreducible $g_1$-twisted $V$-module. The ``$g_2$-twisted'' action
\[
Y^{g_2 \cdot M}(u, z) = Y(g_2u, z)
\]
of $V$ on $M$ defines a new structure of $g_1$-twisted $V$-module, which we denote $g_2 \cdot M$. If $g_2 \cdot M \cong M$ then we say $M$ is \emph{$g_2$-invariant}, and we are then able to choose an equivalence $\xi = \xi : g_2 \cdot M \rightarrow M$ of $g_1$-twisted $V$-modules. In other words
\begin{align}\label{intertwiner.defn}
g_2(u)_{n} = \xi^{-1} u_{n} \xi \quad \text{for all $u \in V^\eps$, $n \in \eps+\Z$}.
\end{align}
We define the \emph{$g_2$-twisted supertrace function} with respect to $\xi$ of $u \in V^0$ on the $g_2$-invariant $g_1$-twisted positive energy $V$-module $M$ to be
\[
S_{M, g_2, \xi}(\tau | u) = \str_M u_0 \xi q^{L_0 - c/24},
\]
wherever the right hand side converges.

In order to describe modular invariance of supertrace functions we must recall the definition of Zhu's modes.
\begin{defn}\label{ZhuModes}
Let $(V, \om)$ be a conformal vertex algebra (with rational conformal weights), and let $\phi(t) = e^{2\pi i t}-1$. Then
\[
Y[u, z] = Y(e^{2\pi i z L_0} u, \phi(z)), \quad \text{and} \quad \widetilde{\om} = (2\pi i)^2 \left[ \om - \frac{c}{24} \vac \right].
\]
We also write $Y[u, z] = \sum_{n \in \Z} u_{([n])} z^{-n-1}$.
\end{defn}
Explicitly
\begin{align}\label{ZhuModesFormula}
\begin{split}
u_{([n])}
&= \res_\nu \nu^n Y(e^{2\pi i \nu L_0}u, e^{2\pi i \nu}-1) d\nu \\
&= (2\pi i)^{-n-1} \res_\xi [\log(1+\xi)]^n Y\left( (1+\xi)^{L_0 - 1} u, \xi \right) d\xi,
\end{split}
\end{align}
and
\begin{align}\label{ZhuConformalZeroMode}
L_{[0]} = L_0 - \sum_{j=1}^\infty \frac{(-1)^j}{j(j+1)} L_j,
\quad \text{where} \quad Y[\widetilde{\om}, z] = \sum_{n \in \Z} L_{[n]} z^{-n-2}.
\end{align}

If $(V, Y(-, z), \om)$ has integer conformal weights, then $(V, Y[-, z], \widetilde{\om})$ is again a conformal vertex algebra. Indeed the two conformal vertex algebra structures are seen to be isomorphic because of Huang's change of coordinate formula, which we now recall. With $\rho \in \C^\times t + t^2\C[[t]]$ we associate the linear endomorphism $R(\rho)$ of $V$ defined by
\[
R(\rho) = \exp\left( -\sum_{j=1}^\infty v_j L_j\right) v_0^{-L_0}, \quad \text{where} \quad \rho(t) = \exp\left( \sum_{j=1}^\infty v_j t^{j+1} \partial_t \right) v_0^{t\partial_t} \cdot t.
\]
For $z \in \C$ the series $\rho_z \in \C^\times t + t^2\C[[t]]$ is defined by $\rho_z(t) = \rho(z+t)-\rho(z)$. Huang \cite{HuangConfBook} proved the formula
\[
Y(u, z) = R(\rho) Y(R(\rho_z)u, \rho(z)) R(\rho)^{-1}
\]
which is basic to the geometric approach to vertex algebras explained in \cite{FBZ04}. If we take $\phi$ as in Definition \ref{ZhuModes} and put $R = R(\phi)$ then
\[
Y[u, z] = R^{-1} Y(R^{-1}u, z) R.
\]
So indeed $V, Y[-, z]$ is isomorphic to $V, Y(-, z)$ via $R$. One easily checks $\widetilde{\om} = R \om$.

We now recall the main theorem of {\cite{JVE13}}, which is a generalisation of Dong, Li, and Mason's {\cite[Theorem 1.3]{DLM00}} to the case of vertex (super)algebras graded by rational conformal weights.
\begin{thm}[{\cite[Theorem 1.3 and Remark 5.2]{JVE13}}]\label{JVEtheorem}
Let $(V, \om)$ be a $\Q$-graded conformal vertex algebra and let $G = \left<g_1\right>$ be a cyclic group of automorphisms of $(V, \om)$ of finite order $N$.
\begin{itemize}
\item Let $V^G$ denote the $G$-invariant subalgebra of $V$, and $W$ the direct sum of the nontrivial eigenspaces of $g_1$. Suppose $V = V^G \oplus W$ is relatively cofinite. Let $g_2 \in G$ and let $M$ be a $g_2$-invariant irreducible positive energy $g_1$-twisted $V$-module. Then for each $u \in V^G$ the series defining $S_{M, g_2, \xi}$ converges absolutely and uniformly on compact sets to a holomorphic function of $\tau \in \HH$.

\item Suppose further that $V$ is $g$-rational for each $g \in G \backslash \{1\}$. For $i, j \in \Z / N\Z$ with $i \neq 0$, let $X(i, j)$ denote the \textup{(}finite\textup{)} set of irreducible $g_1^j$-invariant $g_1^i$-twisted $V$-modules, and let $\CC(i, j)$ denote the vector space spanned by $S_{M, g_1^j, \xi} : V^G \times \HH \rightarrow \C$ as $M$ ranges over $X(i, j)$. If $(i', j') = (i, j) \cdot A$ where $A = \twobytwo abcd \in SL_2(\Z)$, then under the action
\[
[S \cdot A](\tau, u) := S\left( \frac{a\tau+b}{c\tau+d} \bigg| (c\tau+d)^{-L_{[0]}} u \right)
\]
the vector space $\CC(i, j)$ is mapped isomorphically to $\CC(i', j')$.
\end{itemize}
\end{thm}

\begin{rem}
If $V^G$ is $\Z$-graded (as will be the case in this article), then the action of $SL_2(\Z)$ on twisted supertrace functions is a representation.
\end{rem}

\subsection{Li's Operators}\label{subsec.shift}

Let $(V, \om)$ be a $\Z$-graded conformal vertex algebra. Let $h \in V_1$ be an even vector satisfying the Heisenberg $\la$-bracket relation
\begin{align}\label{Li.section.lambda.relation}
[h_\la h] = \text{(const.)} \la
\end{align}
(the value of the constant is unimportant at the moment). Suppose further that $h_{(0)}$ acts semisimply on $V$ with eigenvalues in a lattice, i.e., a discrete subset, $Q \subset (\Q h)^*$.

We now recall the operator series $\D(h, z)$ used by Li to `shift' between differently twisted $V$-modules.
\begin{defn}\label{shift.operator.definition}
For $h \in V$ as above, let
\begin{align}\label{shift.operator.fla}
\D(h, z) = z^{h_{(0)}} \exp \sum_{k = 1}^\infty \frac{(-z)^{-k}}{-k} h_{(k)}.
\end{align}
This expression makes sense on untwisted $V$-modules, and more generally on $g$-twisted $V$-modules whenever $h \in V^g$. The \emph{shifted module} $h * M$ of a $V$-module $M$ is defined to be $M$ as a vector space, equipped with the vertex operation
\[
Y^{h * M}(u, z) = Y^M(\D(h, z) u, z).
\]
\end{defn}
The following theorem is due to Li (under conditions weaker than (\ref{Li.section.lambda.relation}) actually).
\begin{thm}[\cite{Li95}, Proposition 5.4]\label{LiTheorem}
Let $g$ be a finite order automorphism of $V$, let $h \in V_1^g$ be as above, and let $M$ be a $g$-twisted $V$-module. Then $h * M$ is a $g e^{-2\pi i h_0}$-twisted $V$-module.
\end{thm}

For later use we recall some special cases of the action of shift operators. Suppose $h, h' \in V_1$ satisfy $h_{(0)}h' = 0$. Then we have
\begin{align}
\D(h, z)h' &= h' + \left<h, h'\right> \vac z^{-1} \label{AppliedShifts1} \\
\text{and} \quad \D(h, z)\om &= \om + h z^{-1} + \frac{1}{2} \left<h, h\right> \vac z^{-2}. \label{AppliedShifts2}
\end{align}
Let us write
\[
Y(\D(h, z)u, z) = \sum_{n \in \Z} \widehat{u}_n z^{-n-\D(u)}.
\]
One easily verifies that
\begin{align}\label{zeromode.shift}
\widehat{u}_n = [\D(h, 1)u]_0 = \left[ \exp \left( \sum_{n=1}^\infty \frac{(-1)^k}{-k} h_{(k)} \right) u \right]_0
\end{align}
whenever $h_{(0)}u = 0$. Hence
\begin{align}
\widehat{h}'_0 &= h'_0 + \left<h, h'\right> \\
\text{and} \quad
\widehat{L}_0 &= L_0 + h_0 + \frac{1}{2}\left<h, h\right>.
\end{align}

\begin{lemma}\label{twist=shift}
Let $h \in V_1$ be as above, and let $M$ be an $\exp\left(-2\pi i h_0\right)$-twisted $V$-module. Then
\begin{enumerate}
\item The module $M$ may be written as $h * M^0$ for some untwisted $V$-module $M^0$,

\item Let $\eps \in \Q$ and consider the automorphism $g_2 = \exp\left(-2\pi i \eps h_0\right)$ of $V$, and $\xi = \exp\left(+2\pi i \eps h_0\right)$ considered as an automorphism of $M$ via the identification $M \cong h * M^0$ above. Then \textup{(}\ref{intertwiner.defn}\textup{)} is satisfied. In particular $M$ is a $g_2$-invariant module.
\end{enumerate}
\end{lemma}

\begin{proof}
Part (1) is an immediate consequence of Theorem \ref{LiTheorem}. Part (2) is a simple computation. Indeed $\exp\left(-2\pi i \eps h_0\right)$ commutes with $\D(h, z)$, so we have
\begin{align*}
Y^{M}(e^{-2\pi i \eps h_0} u, z)
&= Y^{M^0}(\D(h, z) e^{-2\pi i \eps h_0} u, z)
= Y^{M^0}(e^{-2\pi i \eps h_0} \D(h, z) u, z) \\
&= e^{-2\pi i \eps h_0} Y^{M^0}(\D(h, z) u, z) e^{+2\pi i \eps h_0}
= e^{-2\pi i \eps h_0} Y^{M}(u, z) e^{+2\pi i \eps h_0}.
\end{align*}
Thus $\xi = \exp\left(+2\pi i \eps h_0\right)$ provides the intertwining map that we need.
\end{proof}

{
For $h \in V_1$ it is convenient to denote
\[
I(h) = \sum_{j=1}^\infty \frac{(-1)^j}{-j} h_j.
\]
Observe that
\begin{align}\label{compareZhuLi}
I(h) = (2\pi i)^2 h_{([1])} \quad \text{and} \quad \D(h, 1) = \exp{I(h)}.
\end{align}
}

\section{Preliminaries on Lie Algebras}\label{liealgebras}

\subsection{Lie Algebras and Affine Vertex Algebras}

Let $\ov{\g}$ be a finite dimensional simple Lie algebra over $\C$ of rank $\ell$. We fix a Cartan subalgebra $\ov{\h}$ and a triangular decomposition $\ov{\g} = \ov{\n}_- \oplus \ov{\h} \oplus \ov{\n}_+$, with Borel subalgebra $\ov{\mb} = \ov{\h} + \ov{\n}_+$. We then have the set $\ov{\D} \subset \ov{\h}^*$ of roots, and its subsets $\ov{\D}_+$ of positive roots, and $\ov{\Pi} = \{\ov{\al}_1, \ldots, \ov{\al}_\ell\}$ of simple roots. We denote by $\ov{Q}$ the root lattice $\Z \ov{\D}$.

There is a unique up to scaling nondegenerate invariant bilinear form on $\ov\g$, which induces a form on $\ov\h^*$. The roots come in one or two norms, and the \emph{lacing number} $r^\vee \in \{1, 2, 3\}$ is the ratio between these norms. We denote by $\ov\D_{\text{long}}$ (resp.\ $\ov\D_{\text{short}}$) the set of long (resp.\ short) roots, and we normalise the form $(\cdot, \cdot)$ so that the long roots have norm $2$. We denote by $\nu$ the corresponding identification $\ov{\h} \rightarrow \ov{\h}^*$.

We let $\ov{\theta}$ denote the highest root with respect to the height function $\height : \sum_i k_i \ov{\al}_i \mapsto \sum k_i$ on $\ov{Q}$. It is a long root. We similarly denote by $\ov\theta_{\text{short}}$ the highest of the short roots.

The simple coroots $\ov{\al}_i^\vee \in \ov\h$ are by definition $\ov{\al}_i^\vee = 2 \nu^{-1}(\ov{\al}_i) / (\ov{\al}_i, \ov{\al}_i)$. We denote by $\ov{Q}^\vee$ the coroot lattice $\Z \ov{\D}^\vee$. The coroots come in one or two norms, namely $2$ and $2r^\vee$. We denote by $\ov\D_{\text{long}}^\vee$ (resp.\ $\ov\D_{\text{short}}^\vee$) the set of long (resp.\ short) coroots. We let $\ov{\theta}^\vee = \ov{\theta}^\vee_{\text{long}}$ denote the highest coroot with respect to the height function $\height^\vee : \sum_i k_i \ov{\al}_i^\vee \mapsto \sum k_i$ on $\ov{Q}^\vee$, and $\ov\theta_{\text{short}}^\vee$ the highest of the short coroots. In fact $\ov{\theta}^\vee_{\text{short}} = \nu^{-1}(\ov{\theta})$ and $\ov{\theta}^\vee = r^\vee \nu^{-1}(\ov{\theta}_{\text{short}})$.

The weight lattice $\ov{P} \subset \ov{\h}^*$ is the natural dual of $\ov{Q}^\vee \subset \ov\h$. The fundamental weights $\ov{\La}_1, \ldots, \ov{\La}_\ell$, which form a basis of $\ov P$, are by definition dual to the simple coroots. Similarly the fundamental coweights $\ov{\La}^\vee_i$ are dual to the simple roots $\ov{\al}_i$. We define $\ov{P}_+ = \Z_+ \{\ov{\La}_1, \ldots, \ov{\La}_\ell\}$.

The marks $a_i$ and comarks $a_i^\vee$ ($i = 1, \ldots, \ell$) are defined by $\ov{\theta} = \sum_i a_i \ov{\al}_i$, and $\nu^{-1}(\ov{\theta}) = \sum_i a_i^\vee \ov{\al}_i^\vee$. The dual Coxeter number is $h^\vee = 1 + \sum_i a_i^\vee$. We have the relation $a_i \ov{\al}_i = a_i^\vee \nu(\ov{\al}_i^\vee)$. The Weyl vector is $\ov{\rho} = \sum_{i=1}^\ell \ov{\La}_i = \frac{1}{2}\sum_{\ov{\al} \in \ov\D_+} \ov{\al}$, and the dual Weyl vector is $\ov{\rho}^\vee = \sum_{i=1}^\ell \ov{\La}_i^\vee = \frac{1}{2}\sum_{\ov{\al} \in \ov\D_+} \ov{\al}^\vee$. Clearly $\heit\al = \al(\ov\rho^\vee)$.


The finite Weyl group $\ov{W}$ is the subgroup of $\aut{\ov{\h}^*}$ generated by reflections $s_i : \ov\la \mapsto \ov{\la} - 2 \left<\ov{\la}, \ov{\al}_i^\vee\right> \ov{\al}_i$ in simple roots.


The (untwisted) affine Kac-Moody algebra associated with $\ov{\g}$ as above is
\begin{align*}
\g &= \ov{\g}[t, t^{-1}] \oplus \C K \oplus \C d \\
[a_m, b_n] &= [a, b]_{m+n} + m \delta_{m, -n} (a, b) K, \quad [K, \g] = 0, \quad [d, a_m] = m a_m,
\end{align*}
where $a_m$ denotes $at^m$. We extend $(\cdot, \cdot)$ to $\h = \ov{\h} + \C K + \C d$ by declaring $(K, \ov\h) = (d, \ov\h) = (K, K) = (d, d) = 0$, and $(K, d)=1$. We denote the restriction of $\la \in \h^*$ to $\ov{\h}$ by $\ov{\la}$. We extend $\nu$ to $\h$ using $(\cdot, \cdot)$ and we write $\delta = \nu(K)$, $\La_0 = \nu(d)$. We also set $\al_0 = \delta - \ov\theta$ and $\al_0^\vee = K - \ov\theta_{\text{short}}^\vee$. The level of a weight $\la$ is $\left<\la, K\right>$.

The affine fundamental weights (dual to the coroot basis $\Pi^\vee = \ov\Pi^\vee \cup \{\al_0^\vee\}$) are $\La_0$ together with $\La_i = \ov{\La}_i + a_i^\vee \La_0$ for $i = 1, \ldots, \ell$. Put $\rho = \sum_{i=0}^\ell \La_i$, so that all $\left<\rho, \al_i^\vee\right> = 1$ and $\left<\rho, K\right> = h^\vee$. We write $P = \Z\{\La_i\}$ the affine weight lattice, as well as $P_+ = \Z_+\{\La_i\}$ and $P_{++} = \Z_{\geq 1}\{\La_i\}$. Then $P^k$, $P_+^k$ and $P_{++}^k$ denote their respective subsets of weights of level $k$. We have $P^k = k\La_0 + \ov P$,  similarly we define $Q^{*, k} = k\La_0 + \ov{Q}^* \subset P^k$ (where $\ov Q^*$ is the dual of $\ov Q$ with respect to $(\cdot, \cdot)$).

Let $\ov{W}$ act trivially on $\delta$ and $\La_0$. Any element $\al \in \ov{\h}$ acts on ${\h}^*$ via
\[
t_\al : \la \mapsto \la + \la(K)\al - \left((\la, \al) + \tfrac{1}{2}(\al, \al)\la(K) \right)\delta.
\]
The affine Weyl group is the semidirect product $W = \ov{W} \ltimes t_{\ov{Q}^\vee}$. The coroot system $\D^\vee$ is actually invariant under the larger \emph{extended affine Weyl group} $\widetilde W = \ov W \ltimes t_{Q^*}$. Let $\widetilde W_+$ be the subgroup of automorphisms that preserve the coroot basis $\Pi^\vee$. Explicitly $\widetilde W_+ = \{\sigma_j\}_{j \in J}$, where $J$ is the set of indices $i \in \{1, \ldots, \ell\}$ such that $a_i = 1$, where $\sigma_j = t_{\ov\La_j}\ov\sigma_j$, and where $\ov\sigma_j \in \ov W$ is as in Definition \ref{sigma.remark} below. We record that $\widetilde W_+ \cong \widetilde W / W \cong Q^* / Q^\vee$.

The set of positive real roots of $\g$ is $\D_+^{\text{re}} = \ov{\D}_+ \cup \{\ov{\al} + n\delta | \ov{\al} \in \ov{\D}, n \in \Z_{>0} \} \subset {\h}^*$, and $\D^{\text{re}} = \D_+^{\text{re}} \cup -\D_+^{\text{re}}$. The set of positive real coroots is
\[
\D_+^{\vee \, \text{re}} = \ov{\D}_+^\vee \cup \{\ov{\al} + nK | \ov{\al} \in \ov{\D}^\vee_{\text{short}}, n \in \Z_{>0} \} \cup \{\ov{\al} + n r^\vee K | \ov{\al} \in \ov{\D}^\vee_{\text{long}}, n \in \Z_{>0} \}.
\]

We now record some material on twisted root systems and the Langlands dual to be used in Section \ref{section.coprin.Smatrix}.

Let $\ov\g$, $\g$, etc. be as above, but now suppose $r^\vee > 1$. We introduce the affine root system of twisted type ${}^\circ\D^\vee$ in $\h^*$ as follows:
\[
{}^\circ\D_+^{\text{re}} = \ov{\D}_+ \cup \{\ov{\al} + \tfrac{n}{r^\vee}\delta | \ov{\al} \in \ov{\D}_{\text{short}}, n \in \Z_{>0} \} \cup \{\ov{\al} + n \delta | \ov{\al} \in \ov{\D}_{\text{long}}, n \in \Z_{>0} \}.
\]
The associated coroot system is determined by
\[
{}^\circ\D^{\vee \, \text{re}}_+ = \ov{\D}_+^\vee \cup \{\ov{\al} + nK | \ov{\al} \in \ov{\D}^\vee, n \in \Z_{>0} \}.
\]
We define $\{{}^\circ\La_i\}$ to be the dual basis to $\{{}^\circ\al_i^\vee\}$ and ${}^\circ P_+ = \Z_+\{{}^\circ\La_i\}$. We also define ${}^\circ\rho = \sum_{i=0}^\ell {}^\circ\La_i$, so that $\left<{}^\circ\rho, {}^\circ\al_i^\vee\right> = 1$ for $i = 0,1,\ldots, \ell$.

In the following table we record the types of $\D$ and ${}^\circ\D$.
\[
\renewcommand{\arraystretch}{1.2}
\begin{array}{c|c|c}
\ov \D & \D & {}^\circ\D \\
[0.5mm] \hline 
& & \\
[-4mm]
B_\ell & B_\ell^{(1)} & D_{\ell+1}^{(2)}  \\
C_\ell & C_\ell^{(1)} & A_{2\ell-1}^{(2)}  \\
F_4 & F_4^{(1)} & E_6^{(2)} \\
G_2 & G_2^{(1)} & D_4^{(3)}  \\
\end{array}
\]
Note that the normalisation of $(\cdot,\cdot)$ we have adopted for the twisted root system ${}^\circ\D$ differs from that used in {\cite[Section 6.4]{KacIDLA}}; there the roots have norms $2$ and $2r^\vee$.

The Weyl group of ${}^\circ\D$ is ${}^\circ W = \ov{W} \ltimes t_{\ov{Q}}$, and ${}^\circ\D^\vee$ is invariant under the extended affine Weyl group
\[
{}^\circ\widetilde{W} = \ov W \ltimes t_{P}.
\]
Let ${}^\circ\widetilde W_+$ be the subgroup of automorphisms that preserve the coroot basis ${}^\circ\Pi^\vee = \ov\Pi^\vee \cup \{K - \ov\theta^\vee_{\text{long}}\}$. Explicitly ${}^\circ\widetilde W_+ = \{\sigma_j\}_{j \in {}^LJ}$, where ${}^LJ$ is the set of indices $i \in \{1, \ldots, \ell\}$ such that ${}^La_i = 1$, where $\sigma_j = t_{\ov\La_j}{}^L\ov\sigma_j$, and where ${}^L\ov\sigma_j \in \ov W$ is as in Definition \ref{sigma.remark} below. We record that ${}^\circ\widetilde W_+ \cong {}^\circ\widetilde W / {}^\circ W \cong P / Q$.

We recall the \emph{Langlands dual} ${}^L\ov\D$ to the finite root system $\ov\D$. Explicitly
\[
{}^L\ov\D = \{ {}^L\al | \al \in \ov\D \} \quad \text{where} \quad {}^L\al := \frac{2}{\sqrt{r^\vee}(\al, \al)} \al.
\]
\begin{defn}\label{sigma.remark}
If $a_j=1$ then the set $\{-\ov\theta, \ov\al_1, \ldots, \ov\al_\ell\} \backslash \{\ov\al_j\}$ is a root basis of $\ov\D$. The Weyl group acts simply transitively on root bases. Define $\ov \sigma_j$ to be the unique element of $\ov W$ that sends $-\ov\theta$ to $\ov\al_j$, and permutes the simple roots other than $\ov\al_j$. If ${}^La_j=1$ then $\{-{}^L\ov\theta_{\text{short}}, {}^L\ov\al_1, \ldots, {}^L\ov\al_\ell\} \backslash \{{}^L\ov\al_j\}$ is a root basis of ${}^L\ov\D$, and we define ${}^\circ\ov\sigma_j \in {}^\circ\ov W = \ov W$ in the same way.
\end{defn}

Later we shall require the following lemma.
\begin{lemma}\label{very.technical}
Let $\ov\D$ be a finite type root system of rank $\ell$. And let $J$, $\ov\sigma_j$, ${}^LJ$, and ${}^L\ov\sigma_j$ be as in Definition \ref{sigma.remark}.
\begin{enumerate}
\item \label{VTpart1} The weights $\{\ov\La_j\}_{j \in J}$ represent $\ov Q^*$ modulo $\ov Q^\vee$. The weights $\{\ov\La_j\}_{j \in {}^LJ}$ represent $\ov P$ modulo $\ov Q$.

\item \label{VTpart2} For $j \in J$ it holds that $\ov\sigma_j\ov\La_j = -\ov\La_j$, and for $j \in {}^LJ$ it holds that ${}^L\ov\sigma_j\ov\La_j = -\ov\La_j$.
\end{enumerate}
\end{lemma}

\begin{proof}
(1) For $j \in J$ the claim is well known and appears in {\cite[Section 1.1]{FKW}} for example. Now let $j \in {}^LJ$. The set $\{{}^L\ov\La_j\}_{j \in {}^LJ}$ represents
\[
({}^L\ov Q)^* / {}^L\ov Q^\vee
= ({(1/\sqrt{r^\vee})}\ov Q^\vee)^* / {}^L\ov Q^\vee
= (\sqrt{r^\vee}\ov P)/(\sqrt{r^\vee}\ov Q).
\]
Hence $\{(1/\sqrt{r^\vee}){}^L\ov\La_j\}_{j \in {}^LJ}$ represents $\ov P/\ov Q$. Since
\[
{}^L\al_i^\vee
= \frac{2}{({}^L\al_i, {}^L\al_i)}{}^L\al_i
= \frac{{}^La_i}{{}^La_i^\vee}{}^L\al_i
= \frac{{}^La_i}{\sqrt{r^\vee} {}^La_i^\vee} \al_i^\vee,
\]
our representatives are
\[
(1/\sqrt{r^\vee}){}^L\ov\La_j = ({{}^La_j^\vee}/{{}^La_j}) \ov\La_j = {}^La_j^\vee \ov\La_j.
\]
If $j \in {}^LJ$ then ${}^La_j^\vee = 1$. This is obviously true for simply laced cases, vacuously true for types $F_4$ and $G_2$, and is directly confirmed in the remaining cases $B_\ell$ and $C_\ell$.

(2) We observe
\[
(\ov\sigma_j^{-1}\ov\La_j^\vee, \al_i) = (\ov\La_j^\vee, \ov\sigma_j\al_i) = \left\{\begin{array}{ll} -1 & \text{if $i=j$} \\ 0 & \text{if $i \neq 0, j$} \end{array}\right..
\]
Thus $\ov\sigma_j^{-1}\ov\La_j^\vee = -\ov\La_j^\vee$, and the result follows since $\ov\La_i$ is proportional to $\ov\La_i^\vee$. For $j \in {}^LJ$, we note that the Weyl groups of $\D$ and ${}^L\D$ are canonically identified. The claim now follows immediately from (1).
\end{proof}

Finally we recall definitions relating to affine vertex algebras. Let $\ov\g$, $\g$ be as above, and fix $k \in \C$ (which we assume different than the critical level $-h^\vee$). The \emph{universal affine vertex algebra} is the `vacuum' $\g$-module
\[
V^k(\ov{\g}) = U(\g) \otimes_{U(\ov{\g}[t]+\C K+\C d)} \C v_{k}
\]
where $K$ acts on $v_k$ by the scalar $k$, and $\ov{\g}[t]+\C d$ acts trivially (so $V^k(\ov{\g})$ is naturally a quotient of the Verma module $M(k\La_0)$). The vertex algebra structure on $V^k(\ov{\g})$ is uniquely defined \cite{KacVA} by $Y(a, z) = \sum_{m \in \Z} at^m z^{-m-1}$ for $a \in \ov{\g}$, the vacuum vector $\vac = v_k$, and translation operator $T = L_{-1}$ where $L(z)$ is the Virasoro field associated with the Sugawara vector
\[
\om = \om^\text{Sug} = \frac{1}{2(k+h^\vee)} \sum_{i} a^i_{(-1)} b^i_{(-1)} \vac
\]
(here $\{a^i\}$, $\{b^i\}$ are bases of $\ov{\g}$ dual with respect to $(\cdot, \cdot)$). The central charge of $V^k(\ov\g)$ is
\[
c_k = \frac{k \dim{\ov{\g}}}{k+h^\vee}.
\]
We denote by $V_k(\ov{\g})$ the simple quotient of $V^k(\ov{\g})$.

Let $\la \in \h^*$. The Verma ${\g}$-module is
\[
M(\la) = U(\tilde{\g}) \otimes_{U(\ov\h + \ov\n_+ + t\ov{\g}[t] + \C K + \C d)} \C v_{\la},
\]
where $\ov\h$ acts on $v_\la$ via $\la$, $K$ by $k$ and $\ov\n_+ + t\ov\g[t] + \C d$ acts by $0$. We denote by $L(\la)$ the irreducible quotient of $M(\la)$. Both of these naturally acquire the structure of positive energy $V^k(\ov\g)$-modules, in which $L_0 = -d + h_\la$, where
\[
h_\la = \frac{(\ov\la, \ov\la + 2\ov\rho)}{2(k + h^\vee)}
\]
is the vacuum anomaly.

We denote by $\OO_k$ the category of $\g$-modules $M$ of level $k$ possessing a generalised weight decomposition $M = \bigoplus_{\mu \in \h^*} M_\mu$ such that each $\dim M_\mu < \infty$, and such that the set of weights be contained in a finite union of sets of the form $\mu_i - \Z \D_+$.

\subsection{Characters and Trace Functions}

Let $Y = \HH \times \ov{\h}$ and
\begin{align}\label{Yplus.def}
Y^+ = \left\{ (\tau, h) \in Y \Big| \text{$\Imm{\al(h)} < 0$ for all $\al \in \ov{\D}_+^\vee$ and $\Imm{\ov\theta^\vee(h)} > -\Imm{\tau}$} \right\}.
\end{align}
For any highest weight $\g$-module $M$, the sum
\begin{align}\label{KWchar.gen.def}
\chi_M(\tau, x) = \tr_{M} e^{2\pi i x_0} q^{L_0 - c_k/24}
\end{align}
converges absolutely to a holomorphic function on $Y^+$, and extends to a meromorphic function on $Y$ with possible poles on the hyperplanes
\begin{align}
H_{\al, \om} = \left\{ (\tau, h) \in Y | \al(h) = \om \right\},
\end{align}
for $\al \in \ov{\D}_+^\vee$ and $\om \in \Z + \Z \tau$. Indeed for Verma modules
\[
\chi_{M(\la)}(\tau, x) = q^{h_\la-c_k/24} e^{2\pi i \ov\la(x)} \prod_{n \in \Z_{>0}} \frac{1}{(1-q^n)^\ell} \cdot \prod_{\al \in \ov\D_+^\vee} \prod_{n \in \Z_{+}} \frac{1}{(1 - q^n e^{-2\pi i \al(x)})(1 - q^{n+1} e^{2\pi i \al(x)})}.
\]
On the other hand if $\la$ is dominant integral then $\chi_{L(\la)}$ is holomorphic on $Y$ \cite[Section 12.7]{KacIDLA}.

\subsection{Admissible Weights}

Kac and Wakimoto introduced the notion of admissible weight in \cite{KWPNAS}. The irreducible modules $L(\la)$ for $\la$ admissible of level $k$ are relevant to the representation theory of $V_k(\ov\g)$ and its Hamiltonian reductions (see Section \ref{section.principal.affine} below), and their characters have interesting modular invariance properties.
\begin{defn}
Let $\la \in \h^*$. The associated \emph{integral coroot system} is
\[
R(\la) = \{\al^\vee \in \D^{\vee \, \text{re}} | \left< \la + \rho, \al^\vee \right> \in \Z\}.
\]
The weight $\la$ is called \emph{admissible} if
\[
\text{$\left< \la + \rho, \al^\vee \right> \notin \Z_{\leq 0}$ for all $\al^\vee \in \D_+^{\vee \, \text{re}}$}.
\]
The admissible weight $\la$ is called \emph{$G$-integrable} if
\[
\text{$\left<\la, \al^\vee\right> \in \Z_{\geq 0}$ for all simple roots $\al \in \ov{\Pi}$ of the finite Lie algebra $\ov\g$}.
\]
The admissible weight $\la$ is called \emph{principal admissible} if
\[
\text{$R(\la)$ is isometric to $\D^{\vee \, \text{re}}$}.
\]
 The sets of admissible, $G$-integrable admissible, and principal admissible weights of level $k$ are denoted respectively $\adm^k$, $\adm^k_+$, and $\prin^k$.
\end{defn}
\begin{exmp}
Let $\al_1$, $\al_2$ denote the long and short simple roots of $G_2$, respectively. Then the weight $\la = \frac{1}{3}\La_0 + \frac{1}{2}\La_1 + \La_2$ of $G_2^{(1)}$ (of level $k = 7/3$) is admissible. Since
\[
R(\la) = \left\{ 3nK \pm \al_1^\vee \pm \al_2^\vee | n \in \Z \right\}
\]
is isometric to the affinisation of $A_1 \oplus A_1$, the weight $\la$ is not principal admissible.
\end{exmp}
The classification of all admissible weights was carried out in \cite{KW89}. The classification of $G$-integrable admissible weights is much simpler, and consists of just two cases. Indeed let $k \in \Q$, put $k+h^\vee = p/q$ where $(p, q)=1$, and define $S_{(q)} \subset \D^\vee_+$ by
\begin{align}\label{coroot.system.definition}
\begin{split}
S_{(q)} = \{\ga_0^\vee, \ldots, \ga_\ell^\vee\}, \quad \text{where} \quad \ga_i^\vee &= \text{$\al_i^\vee$ for $i = 1, \ldots, \ell$, and} \\
\ga_0^\vee &= \left\{\begin{array}{ll}
qK - \ov\theta^\vee_{\text{short}} & \text{if $(q, r^\vee)=1$,} \\
qK - \ov\theta^\vee_{\text{long}} & \text{if $(q, r^\vee) \neq 1$.} \\
\end{array}
\right.
\end{split}
\end{align}
If $\la \in \adm^k_+$ then $R(\la)$ is the coroot system with base $S_{(q)}$. If $(q, r^\vee)=1$ then $R(\la)$ is isometric to $\D^\vee$, and if $(q, r^\vee)\neq 1$ then $R(\la)$ is isometric to ${}^\circ\D^\vee$.

Let $\la \in \prin^k$, then $(q, r^\vee)=1$ and {\cite[Lemma 2.1]{KW89}} implies that $R(\la) = y(S_{(q)})$ for some $y \in \widetilde{W}$.

\begin{defn}\label{coprin.definition}
The admissible weight $\la$ is called \emph{coprincipal admissible} if $R(\la) = y(S_{(q)})$ for some $y \in \widetilde{W}$ and some $q$ such that $(q, r^\vee)\neq 1$. The set of coprincipal weights of level $k$ is denoted $\coprin^k$.
\end{defn}

\begin{rem}
If $\la$ is coprincipal admissible then $R(\la)$ is isometric to ${}^\circ\D^\vee$. In {\cite[Table 1]{KW89}} coroot systems $R(\la)$ for admissible $\la$ are classified up to the action of $\widetilde{W}$. There is a unique $\widetilde{W}$-orbit of such coroot bases equivalent to ${}^\circ\D^\vee$ in types $B_\ell$, $F_4$ and $G_2$. In type $C_\ell$ there are two such $\widetilde{W}$-orbits. In other words, $\la \in \coprin^k$ if and only if $R(\la)$ is isometric to ${}^\circ\D^\vee$ in types $B_\ell$, $F_4$ and $G_2$, but not in type $C_\ell$.
\end{rem}

The number $k \in \Q$ is called an \emph{admissible}, \emph{principal admissible}, or \emph{coprincipal admissible number} if $k\La_0$ lies in $\adm^k$, $\prin^k$, $\coprin^k$ respectively.

We see the importance of admissible weights from the perspective of vertex algebras in the following result.
\begin{thm}[{\cite[Main Theorem]{A12catO}}]\label{rational.category.O.theorem}
Let $k \in \Q$ be a principal \textup{(}resp.\ coprincipal\textup{)} number for $\ov\g$. The $V^k(\ov\g)$-module $L(\la)$ descends to a module over the simple quotient $V_k(\ov\g)$ if and only if $\la \in \prin^k$ \textup{(}resp.\ $\la\in \coprin^k$\textup{)}. Furthermore any $V_k(\ov\g)$-module from category $\OO_k$ is completely reducible.
\end{thm}

Having put $k+h^\vee=p/q$ as above, let $\phi : \h^* \rightarrow \h^*$ be the isometry
\begin{align}\label{isometry.phi}
\phi(\La_0) = (1/q)\La_0, \quad \phi(\delta) = q\delta, \quad \phi|_{\ov\h^*} = 1.
\end{align}
For the rest of this section we assume $k \in \Q$ to be a principal admissible number. The weights $\la \in \adm^k$ satisfying $R(\la) = y(S_{(q)})$ for some fixed $y \in \widetilde{W}$ are precisely
\[
\la = y(\phi(\nu)) - \rho,
\]
where $\nu$ ranges over the set of regular elements of $P_+^p$. Let $y = t_\beta \ov y$ where $\ov y \in \ov W$ and $\beta \in \ov Q^*$. In this way a triple $(\nu, \ov y, \beta)$ is associated with $\la \in \adm^k$.

In {\cite{KWPNAS}} and {\cite{KW89}} Kac and Wakimoto established modular properties for the characters $\chi_\la = \chi_{L(\la)}$ of irreducible $\g$-modules $L(\la)$ of principal admissible highest weight. In the next section we derive a similar result in the coprincipal case.
\begin{thm}[\cite{KW89}, Theorem 3.6]\label{KWSmatrix}
Let $k \in \Q$ be a principal admissible number, and put $k+h^\vee = p/q$ where $(p, q) = 1$. Then the $\C$-linear span of the set $\{\chi_\la | \la \in \prin^k\}$ is invariant under the action
\[
\left[ f \cdot \twobytwo abcd \right](\tau, x) = \exp{\left[ \frac{\pi i k c (x, x)}{c\tau+d} \right]} f\left( \frac{a\tau+b}{c\tau+d}, \frac{x}{c\tau+d} \right)
\]
of $SL_2(\Z)$ on functions of $(\tau, x) \in \HH \times \ov{\h}$. Furthermore the $S$-matrix $\{a({\la, \la'})\}$, defined by
\[
\chi_\la \cdot S = \sum_{\la' \in \prin^k} a({\la, \la'}) \chi_{\la'},
\]
is given explicitly by
\begin{align*}
a({\la, \la'})
= {} & \frac{i^{|\ov{\D}_+|}}{|\ov P/pq \ov Q^\vee|^{1/2}} e^{-2\pi i \left[ (\nu | \beta') + (\nu' | \beta) + \frac{p}{q} (\beta | \beta') \right]} \eps(\ov{yy}')
\sum_{w \in \ov{W}} \varepsilon(w) e^{-\frac{2\pi i q}{p} (w(\nu) | \nu')}.
\end{align*}
where $(\nu, \ov y, \beta)$ is a triple associated with $\la$ as above, and $(\nu', \ov y', \beta')$ is similarly associated with $\la'$.
\end{thm}

\section{Coprincipal $S$-matrix}\label{section.coprin.Smatrix}

In Theorem \ref{theorem1.affine} below we establish $SL_2(\Z)$-invariance for the affine vertex algebra {$V_k(\ov\g)$} at (principal or coprincipal) admissible level $k$. In this section we compute the associated $S$-matrix in the coprincipal case. The proof follows the pattern of \cite{KW89} (see also {\cite[Chapter 3]{WakimotoBook}}), but several adaptations to the coprincipal case must be made.

We assume then that $k$ is a coprincipal admissible number, and we write $k+h^\vee = p/q$ so that $(p, q) = 1$ and $r^\vee$ divides $q$. Defining $S_{(q)}$ and $\phi$ as in (\ref{coroot.system.definition}) and (\ref{isometry.phi}), we have $\phi^*(S_{(q)}) = {}^\circ\D^\vee$. The admissible weights with coroot basis $S_{(q)}$ are exactly those $\la \in \h^*$ such that
\[
\left< \la + \rho, \ga_i^\vee \right> = n_i+1 \in \Z_{\geq 1} \quad i = 0, \ldots, \ell.
\]
Equivalently $\phi^{-1}(\la+\rho)$ is strictly dominant integral relative to ${}^\circ\D^\vee$. The level of this weight is
\[
\left< \phi^{-1}(\la+\rho), K \right> = \left< \la+\rho, qK \right> = q(k+h^\vee) = p.
\]
Hence the coprincipal admissible weights with coroot system $y(S_{(q)})$ are exactly the weights
\begin{align*}
\la = y(\phi(\nu)) - \rho \quad \text{for regular $\nu \in {}^\circ P_{+}^p$}.
\end{align*}
We observe that $p\La_0 + \ov\la \in {}^\circ P_{+}^p$ if and only if $\la \in \ov P$ satisfies $\left<\ov\la, \ov\al_i^\vee\right> \in \Z_+$ for $i = 1,\ldots, \ell$ and $\left<\ov\la, \ov\theta^\vee\right> \leq p$. Therefore it is equivalent to write
\begin{align*}
\la = y(\phi(\nu+{}^\circ\rho)) - \rho \quad \text{for $\nu \in {}^\circ P_{+}^{p - h}$}.
\end{align*}
Let $y = t_\beta \ov y$ where $\beta \in \ov Q^*$ and $\ov y \in \ov W$. In this way a triple $(\nu, \ov y, \beta)$ is associated with $\la \in \coprin^k$.

The normalised character $\chi_\la$ of a highest weight $\g$-module of arbitrary admissible highest weight $\lambda$ is given by {\cite[Theorem 1]{KWPNAS}}
\[
{\chi_\la(h)} = \frac{A_{\la+\rho}(h)}{A_\rho(h)},
\]
where by definition
\[
A_{\la}(h) = e^{-\frac{(\la, \la)}{2\left<\la, K\right>} (\delta, h)} \sum_{w \in W(\la)} \eps(w) e^{(w(\la), h)}.
\]
Here $W(\la)$ is the 
integral Weyl group of $\la$,
that is the
subgroup of $W$ generated by reflections in roots of $R(\la)$.

Now, $(\la+\rho, h) = (y(\phi(\nu)), h) = (\nu, \phi^{-1} y^{-1} h)$, and in the coordinates
\[
(\tau, z, t) \equiv 2\pi i \left( -\tau \La_0 + \sum z_i \al_i + t\delta \right)
\]
one has
\[
\phi^{-1}y^{-1}(\tau, z, t) = q \ov y^{-1} t_{-\beta/q}(\tau, z/q, t/q^2).
\]
This change of coordinates intertwines the summation on $W(\la) = y W(S_{(q)})y^{-1}$ with a summation on $W({}^\circ\D^\vee) = \ov W \ltimes \ov Q$. Thus we have:
\begin{lemma}\label{coprin.coord.change}
Let $\la$ be a coprincipal admissible weight, and $\nu$, $\ov y$, $\beta$ as above.
\[
\chi_\la(\tau, z, t) = \frac{A_{\nu}(q \ov y^{-1} t_{-\beta/q}(\tau, z/q, t/q^2))}{A_\rho(\tau, z, t)}.
\]
\end{lemma}
To determine modular properties of the $\chi_\la$, we wish to express the numerator of the right hand side in Lemma \ref{coprin.coord.change} in terms of the theta functions
\begin{align}\label{coprin.theta.def}
\Theta_\mu(h) = e^{-\frac{(\mu, \mu)}{2\left<\mu, K\right>}(\delta, h)} \sum_{t \in \ov Q} e^{(t(\mu), h)}.
\end{align}
Since
\[
\Theta_\mu(qh) = \Theta_{q\mu}(h) \quad \text{and} \quad \Theta_{t_\beta \mu}(h) = \Theta_\mu(t_{-\beta} h),
\]
it follows that
\begin{align}\label{A.to.B}
\begin{split}
A_{\nu}(q \ov y^{-1} t_{-\beta/q}(\tau, z/q, t/q^2)
%
&= \eps(\ov y) \sum_{\ov w \in \ov W} \eps(\ov w) \Theta_{\ov w(\nu)}\left(q t_{-\beta/q}(\tau, z/q, t/q^2)\right) \\
&= \eps(\ov y) \sum_{\ov w \in \ov W} \eps(\ov w) \Theta_{q t_{\beta/q} \ov w(\nu)}(\tau, z/q, t/q^2).
\end{split}
\end{align}
We observe that
\[
q t_{\beta/q} \ov w(\nu) = q \ov w(\nu) + p \beta \in q P + Q^* \subset Q^*,
\]
and is a weight of level $pq$.

The following proposition is {\cite[Theorem 13.5]{KacIDLA}} (where the stronger hypothesis that $L$ be integral is implicit but not used).
\begin{prop}\label{General.Theta.Trans}
Let $L$ be a positive definite lattice of rank $\ell$, and let $m \in \Z_+$ be such that the rescaled lattice $\sqrt{m}L$ is integral. Let $\Theta_\mu$ be defined as in \textup{(}\ref{coprin.theta.def}\textup{)}. For any $\mu \in L^*$ one has
\[
\Theta_\mu\left( -\frac{1}{\tau}, \frac{z}{\tau}, t - \frac{(z, z)}{2\tau} \right)
= \frac{(-i\tau)^{\ell/2}}{|L^* / mL|^{1/2}} \sum_{\mu' \in L^* \bmod{m L}} e^{-2\pi i (\mu, \mu') / m} \Theta_{\mu'}.
\]
\end{prop}


Let us write $B_\la(\tau, z, t)$ for (\ref{A.to.B}). By Proposition \ref{General.Theta.Trans} we have
\begin{align}\label{B.S.reduction}
[{B_\la} \cdot S](\tau, z, t)
&= \eps(\ov y) \sum_{\ov w \in \ov W} \eps(\ov w) {\Theta_{q\ov w(\nu)+p\beta}}(-1/\tau, z/\tau, 1/q^2(t - (z, z)/2\tau)) \nonumber \\
&= (-i)^{\ell/2} |\ov Q^* / pq\ov Q|^{-1/2} \eps(\ov y) \sum_{\ov w \in \ov W} \eps(\ov w) \sum_{\mu' \in \ov Q^* \bmod{pq \ov Q}} e^{-\frac{2\pi i}{pq} (q\ov w(\nu)+p\beta, \mu')} \Theta_{\mu'}(\tau, z/q, t/q^2).
\end{align}
For all $\nu \in \ov P$ and $\al \in \ov\D$ one has $r_\al(\nu) - \nu = \left<\nu, \al^\vee\right>\al \in \ov Q$, and so for all $\ov w \in \ov W$ one has $\ov w(\nu) - \nu \in Q$. It follows that for all $\beta \in \ov Q^*$ and $\nu \in \ov P$ one has $(\ov w(\nu), \beta) - (\nu, \beta) \in (\ov Q, \ov Q^*) = \Z$. Thus modulo $\Z$ it holds that
\begin{align}\label{inner.product}
(q\ov w(\nu)+p\beta, q\ov w'(\nu')+p\beta')
\equiv q^2 (\ov w \nu, \ov w' \nu') + pq\left( (\nu, \beta') + (\nu', \beta) \right) + p^2 (\beta, \beta').
\end{align}

By Lemma {\ref{big.old.technical}\,\,(\ref{BOTmu})} below we are free to substitute $\mu' = q \ov w'(\nu') + p\beta'$ in (\ref{B.S.reduction}). It follows from (\ref{inner.product}) and a standard symmetry argument that only regular $\nu'$ contribute nontrivially to the sum.

What results is a sum over $\ov w \in \ov W$ and over equivalence classes of $(\nu', \ov w', \beta')$. More precisely $\nu'$ runs over a system of representatives of ${}^\circ P^p_+ / {}^\circ \widetilde{W}_+$, $\ov w'$ runs over $\ov W$, and $\beta$ is determined by $\mu'$, $\nu'$ and $\ov w'$. We set $\ov w = \ov w'\ov w''$, and rewrite (\ref{B.S.reduction}) as
\begin{align*}
(-i)^{\ell/2} |\ov Q^* / pq\ov Q|^{-1/2} & \eps(\ov y)
\sum_{\nu', \beta'} e^{-2\pi i[ (\nu, \beta') + (\nu', \beta) + \frac{p}{q}(\beta, \beta')} \\
& \times \sum_{\ov w'' \in \ov W} \eps(\ov w'') e^{-2\pi i\frac{q}{p}(\ov w'' \nu, \nu')}
\sum_{w' \in \ov W} \eps(w') \Theta_{q \ov w'(\nu') + p\beta'}(\tau, z/q, t/q^2).
\end{align*}
Now we use $\beta'$ to determine an element $y' \in \widetilde W$ as in Lemma \ref{big.old.technical}, and we put
\[
\la' = y'\phi(\nu')-\rho.
\]
We now have
\[
A_{\la'}(\tau, z, t) = \eps(\ov y') \sum_{w' \in \ov W} \eps(w') \Theta_{q \ov w'(\nu') + p\beta'}(\tau, z/q, t/q^2).
\]
By Lemma {\ref{big.old.technical}\,\,(\ref{BOTla.indep})} the weight $\la'$ depends on $\mu'$, and not on the choice of $\nu'$ in its $\widetilde{W}_+$-orbit.

Recall {\cite[Lemma 13.8 (b)]{KacIDLA}} that
\[
{A_\rho}(-1/\tau, z/\tau, t - (z, z)/2\tau) = (-i)^{\ell/2 + |\ov\D_+|} A_\rho(\tau, z, t).
\]

We thus have the following theorem.
\begin{thm}\label{KWSmatrix.coprin}
Let $k \in \Q$ be a coprincipal admissible number, and put $k+h^\vee = p/q$ where $(p, q) = 1$ \textup{(}so $r^\vee$ divides $q$\textup{)}. Then the $\C$-linear span of the set $\{\chi_\la | \la \in \coprin^k\}$ is invariant under the action
\[
\left[f \cdot \twobytwo abcd \right](\tau, x) = \exp{\left[ \frac{\pi i k c (x, x)}{c\tau+d} \right]} f\left( \frac{a\tau+b}{c\tau+d}, \frac{x}{c\tau+d} \right)
\]
of $SL_2(\Z)$ on functions of $(\tau, x) \in \HH \times \ov{\h}$. Furthermore the $S$-matrix $\{a({\la, \la'})\}$, defined by
\[
\chi_\la \cdot S = \sum_{\la' \in \coprin^k} a({\la, \la'}) \chi_{\la'},
\]
is given explicitly by
\begin{align*}
a({\la, \la'})
= {} & \frac{i^{|\ov{\D}_+|}}{|\ov Q^*/pq \ov Q|^{1/2}} e^{-2\pi i \left[ (\nu | \beta') + (\nu' | \beta) + \frac{p}{q} (\beta | \beta') \right]} \eps(\ov{yy}')
\sum_{w \in \ov{W}} \varepsilon(w) e^{-\frac{2\pi i q}{p} (w(\nu) | \nu')}.
\end{align*}
where $(\nu, \ov y, \beta)$ is a triple associated with $\la$ as above, and $(\nu', \ov y', \beta')$ is similarly associated with $\la'$.
\end{thm}

\begin{rem}
We have related characters of coprincipal admissible $\g$-modules to characters of a certain subset of integrable highest weight ${}^\circ\g$-modules, where ${}^\circ\g$ is the twisted affine Lie algebra with root system ${}^\circ\D$. See also \cite[(3.4)]{KW89}. The characters of integrable highest weight ${}^\circ\g$-modules are detailed in {\cite[Proposition 4.5]{KP84}} and {\cite[Theorem 13.9]{KacIDLA}}. In particular their span is not $SL_2(\Z)$-invariant, and is instead invariant under a congruence subgroup. On the other hand the span of the characters of coprinciple admissible $\g$-modules is invariant under the whole of $SL_2(\Z)$. The difference comes about because (1) these admissible weights correspond to a proper subset of integral weights of ${}^\circ\D$, and (2) the Weyl denominator of $\D$ is $SL_2(\Z)$-invariant, while that of ${}^\circ\D$ is not.
\end{rem}

In the proof of Theorem \ref{KWSmatrix.coprin} we have used the following technical lemmas which are the coprincipal analogue of {\cite[Lemma 3.4]{KW89}}.


\begin{lemma}\label{ga.from.beta}
{Let $\D$ be a root system of type $X_N^{(1)}$ where $X_N$ has lacing number $r^\vee$}, and let $q \in \Z_+$ be a multiple of $r^\vee$. Let $S_{(q)}$ be as above. For any $\beta \in \ov Q^*$, there exists a unique $\ov{y} \in \ov{W}$ and a unique $\ga \in \ov Q$ such that $t_{\beta+q\ga}\ov{y}(S_{(q)}) \subset \D_+^\vee$.
\end{lemma}

\begin{proof}
We note that $S_{(q)}$ is the set of simple coroots for
\[
\D^\vee_{(q)} = \{\al \in R | \left< \La_0, \al \right> \in q\Z\},
\]
(which is a coroot system isomorphic to ${}^\circ\D^\vee$). The canonical imaginary coroot of $\D^\vee_{(q)}$ is $qK$, and the Weyl group is
\[
W_{(q)} = \{\ov w t_{q\al} | \ov w \in \ov W, \al \in \ov Q\} \cong \ov W \ltimes \ov Q.
\]
Let $C$ be the fundamental chamber of $\D^\vee$ and $C_{(q)}$ the fundamental chamber of $\D^\vee_{(q)}$. Let $\xi \in C \subset C_{(q)}$ be a regular element (i.e., $\left<\xi, \al^\vee\right> \notin \Z$ for all $\al^\vee \in \D^\vee$) such that $\left<\xi, K\right> = 1$. Since $\beta \in \ov Q^* \subset \ov P$ we see that $t_{-\beta}(\xi) = \xi-\beta \bmod{\C\delta}$ is also regular. Hence there exists unique $w \in W_{(q)}$ such that $w t_{-\beta}(\xi) \in C_{(q)}$. Writing $w = \ov y^{-1} t_{-q\ga}$, this tells us that
\[
\ov y^{-1} t_{-\beta-q\ga}(\xi) \in C_{(q)},
\]
or rather that
\[
\left< \xi, t_{\beta+q\ga} \ov y(\ga_i) \right> > 0 \quad \text{for $i = 0, \ldots, \ell$.}
\]
But $\xi \in C$ and $t_{\beta+q\ga} \ov y(\ga_i) \in \widetilde{W}\ga_i \subset \D^\vee$. Hence we have $t_{\beta+q\ga} \ov y(\ga_i) \subset \D^\vee_+$.
\end{proof}

\begin{lemma}\label{big.old.technical}
{Let $\D$ be a root system of type $X_N^{(1)}$ where $X_N$ has lacing number $r^\vee$}, and let $q \in \Z_+$ be a multiple of $r^\vee$. Let $S_{(q)}$ be as above, and let $p$ be coprime to $q$.
\begin{enumerate}
\item \label{BOTmu} Any element $\mu \in Q^{*, pq}$ can be written in the form
\begin{align}\label{mutola}
\mu = q\ov w(\nu) + p\beta \quad \text{where $\nu \in {}^\circ P_+^p$, $\beta \in \ov Q^*$, and $\ov w \in \ov{W}$}.
\end{align}

\item \label{BOTdet} If $(\nu, \beta, \ov w)$ is a solution of equation \textup{(}\ref{mutola}\textup{)} with $\nu$ regular, then $\ov w$ and $\beta$ are uniquely determined by $\nu$.


\item \label{BOTnumber} Let $\mu$ be such that $\nu$ is regular in \textup{(}\ref{mutola}\textup{)}, then the equation has precisely $|J|$ solutions $(\nu, \beta, \ov w)$.

\item \label{BOTtransitive.action} For $j \in {}^LJ$ let $\sigma_j = t_{\ov\La_j}\ov{\sigma}_j$ be as in Remark \ref{sigma.remark}. The transformation
\[
\nu_j = \sigma_j \nu, \quad \beta_j = \beta - q\ov w \ov \sigma_j^{-1} \ov \La_j, \quad \ov w_j = \ov w \ov\sigma_j^{-1}
\]
sends solutions to (\ref{mutola}) to solutions, and is transitive.

\setcounter{ForTheListInLemma34}{\value{enumi}}
\end{enumerate}
Now let $(\nu_i, \beta_i, \ov w_i)$ be a solution to \textup{(}\ref{mutola}\textup{)}, and let $\ga_i \in \ov Q$ and $\ov y_i \in \ov W$ be the unique solutions to $t_{\beta_i+q\ga_i}\ov{y}_i(S_{(q)}) \subset \D^\vee_+$ as in Lemma \ref{ga.from.beta}. We define $y_i = t_{\beta_i+q\ga_i} \ov{y}_i \in \widetilde{W}$ for $i = 1, \ldots, |J|$.
\begin{enumerate}
\setcounter{enumi}{\value{ForTheListInLemma34}}
\item \label{BOTS.indep} The root subsystem $S = y_j(S_{(q)})$ depends on $\mu$ but not on $j$.

\item \label{BOTla.indep} The weight $\la = y_j\phi(\nu_j) - \rho$ depends on $\mu$ but not on $j$.

\item \label{BOTla.admissible} If $\nu \in {}^\circ P^p_+$ then $\la$ is admissible.
\end{enumerate}
\end{lemma}

\begin{proof}
(\ref{BOTmu}) Let
\[
\mu = pq\La_0 + \sum_{i=1}^\ell n_i \ov\La_i,
\]
for $i=1,\ldots, \ell$ pick a solution to
\[
n_i' q + n_i'' p = n_i,
\]
and set $\nu_0 = p\La_0 + \sum n_i' \ov\La_i$ and $\beta_0 = \sum n_i'' \ov\La_i$. Then $\nu_0 \in p\La_0 + \ov P$, but since $q\ov P \subset \ov Q^*$ and $(p, q) = 1$ we deduce $p \beta_0 = \mu - q\nu_0 \in \ov Q^*$.

Next we consider $\nu_0$. Every $W({}^\circ\D^\vee)$-orbit on $p\La_0 + \ov P$ meets the positive chamber ${}^\circ P^p_+$, so we let $\nu_0 = w(\nu)$ for some $\nu \in P_+^p$. Now write $w = t_\xi \ov w$ (where $\xi \in \ov Q$, $\ov w \in \ov W$), that is,
\[
\nu_0 = \ov w \nu + p\xi \bmod{\C\delta}.
\]
Now let $\beta = \beta_0 + q\xi$. Then we have
\[
q \ov w \nu + p \beta = q (\nu_0 - p\xi) + p(\beta_0 + q\xi) = q \nu_0 + p \beta_0 = \mu
\]
as desired. Note that $\beta \in \ov Q^*$ still, indeed $q\xi \in q\ov Q \subset q\ov P \subset \ov Q^*$.

(\ref{BOTdet}) The choice of regular dominant $\nu$ made above uniquely specifies the corresponding element $w \in W$, hence also $\beta$ and $\ov w$.

(\ref{BOTnumber}) Let $(\nu, \beta) \in (p\La_0+\ov P) \times \ov Q^*$ be a solution to $\mu = q\nu + p\beta$. All solutions to this equation are of the form $(\nu - p\zeta, \beta + q\zeta)$ for $\zeta \in \ov P$. Let $\zeta \in \ov P$. Since $\nu$ is regular so is $\nu'_0 = \nu - p\zeta$. Hence there exists unique $w = t_{\xi} \ov w \in W({}^\circ\D^\vee) = \ov W \ltimes \ov Q$ such that $\nu' = w^{-1}(\nu'_0) \in {}^\circ P^p_+$. Putting $\beta' = \beta + q\zeta - q\xi$ yields
\[
\mu = q \ov{w}(\nu') + \beta'.
\]
We observe that if $\mu \in Q^*$ then $\beta \in \ov Q^*$ and $\beta' = \beta + q\zeta - q\xi \in \ov Q^* + q\ov P + q\ov Q = \ov Q^*$. Also the class of $\beta' - \beta$ is well defined in $q\ov P / q\ov Q$. Hence we obtain exactly $|\ov P / \ov Q| = |J|$ distinct solutions.

(\ref{BOTtransitive.action}) By direct calculation
\begin{align*}
q \ov w_j(\nu_j) + p \beta_j = q \ov w(\nu)  + p \beta,
\end{align*}
so solutions are sent to solutions. To show transitivity we need to show that the weights $\sigma_j^{-1} \ov \La_j$ are all distinct. This follows from Lemma {\ref{very.technical}\,\,(\ref{VTpart2})}.


(\ref{BOTS.indep}) We recall that the element $\sigma_j^{(q)} := t_{q\ov \La_j} \ov\sigma_j$ maps $S_{(q)}$ into itself. Using the relation $t_{\ov w \al} = w t_\al w^{-1}$ we compute
\begin{align*}
y
= t_{\beta+q\ga}\ov y 
= t_{\beta_j+q\ga} t_{q\ov w \ov\sigma_j^{-1}\ov \La_j} \ov y 
= t_{\beta_j+q\ga} \ov w \ov\sigma_j^{-1} \sigma_j^{(q)} \ov w^{-1} \ov y (\sigma_j^{(q)})^{-1} \sigma_j^{(q)}.
\end{align*}
Since $\ov W$ is normal in $\widetilde W$ we have
\[
\ov w \ov\sigma_j^{-1} \sigma_j^{(q)} \ov w^{-1} \ov y (\sigma_j^{(q)})^{-1} \in \ov W.
\]
Since $\sigma_j^{(q)}$ preserves $S_{(q)}$, the uniqueness property of Lemma \ref{ga.from.beta} implies $y = y_j \sigma_j^{(q)}$. It follows that $y(S_{(q)}) = y_j(S_{(q)})$.

(\ref{BOTla.indep}) By direct calculation (using $y = y_j \sigma_j^{(q)}$) we obtain $y\phi(\nu) = y_j\phi(\nu_j)$.

(\ref{BOTla.admissible}) To check that $\la$ is admissible with base $S$ it suffices to confirm that $\left<\phi(\nu), {}^\circ\ga^\vee_i\right> \in \Z_{\geq 1}$ for $i=0, \ldots, \ell$. By direct calculation $\left<\phi(\nu), {}^\circ\ga_i^\vee\right> = \left<\nu, \al_i^\vee\right>$ for all $i$. So the claim follows from $\nu \in {}^\circ P_+^p$.

\end{proof}


\section{Modular Invariance of Vertex Algebra Characters}\label{main.computation}

Let $(V, \om)$ be a conformal vertex algebra of central charge $c$, graded by integral conformal weights. Let $h \in V_1$ be a vector satisfying the $\la$-bracket relations
\begin{align}\label{main.lambda.relation}
[h_\la h] = 2 \la \quad \text{and} \quad [\om_\la h] = (T+\la)h - p \frac{\la^2}{2}\vac,
\end{align}
i.e.,
\[
[h_{(m)}, h_{(n)}] = 2m \delta_{m, -n} \quad \text{and} \quad [L_{m}, h_{(n)}] = -n h_{(m+n)} - \frac{m^2-m}{2} \delta_{m, -n} p.
\]
Let us assume also that $h_{0}$ induces an eigenspace decomposition of $V$ of the form $V = \bigoplus_{\al \in Q} V^{(\al)}$, where $Q$ is a rank $1$ lattice in $(\Q h)^*$ and $h_{0}u = \al(h) u$ for all $u \in V^{(\al)}$. We write $V^\text{ne} = V^{(0)}$.

We introduce the vector
\[
\om(\sigma) = \om - \tfrac{\sigma}{2} Th,
\]
depending on the parameter $\sigma \in \Q$. The modes of $L(\sigma)(z) = Y(\om(\sigma), z)$ satisfy the commutation relations of the Virasoro algebra with central charge
\[
c(\sigma) = c + 6 \sigma(p-\sigma).
\]

\begin{lemma}\label{lemma.shift.of.shift}
Let $\al \in \Q$ and let $\widehat{L(\sigma)}(z)$ denote the shifted field $((\al-1)\tfrac{\sigma}{2}h) * L(\sigma)(z)$. Then
\[
\widehat{L(\sigma)}_0 - \frac{c(\sigma)}{24} = L(\al\sigma)_0 - \frac{c(\al\sigma)}{24}.
\]
\end{lemma}

\begin{proof}
By direct computation
\begin{align*}
\D(\beta h, z)\om(\sigma)
= \om - \tfrac{\sigma}{2} Th + \beta h z^{-1} + \tfrac{1}{2} \left[ 2\beta^2 - \beta(p-2\sigma) \right] \vac z^{-2}.
\end{align*}
Substituting $\beta = (\al-1)\tfrac{\sigma}{2}$ yields the result, after a short calculation.
\end{proof}

If the grading on $V$ induced by $L_0(\sigma) = L_0 + \tfrac{\sigma}{2}h_0$ is bounded below, then $\om(\sigma)$ is a conformal vector.
\begin{defn}
We say that the conformal vertex algebra $(V, \om)$ is \emph{rational relative to $h$} if for all sufficiently small $\sigma \in \Q_{>0}$ we have $(V, \om(\sigma))$ conformal and rational. We call a positive energy $(V, \om)$-module \emph{$h$-stable} if it remains positive energy as a $(V, \om(\sigma))$-module for all sufficiently small $\sigma \in \Q_{>0}$.
\end{defn}
We assume henceforth that $(V, \om)$ is rational relative to $h$.

We now introduce the finite order automorphism
\begin{align}\label{gsigma.2}
g(\sigma) = \exp\left(- 2\pi i \tfrac{\sigma}{2} h_0\right)
\end{align}
of $V$. It follows from Definition \ref{twisted.mod.def} that $(V, \om(\sigma))$ is a $g(\sigma)$-twisted $V$-module, indeed
\[
e^{-2\pi i L_0(\sigma)} = e^{-2\pi i L_0 - 2\pi i \frac{\sigma}{2} h_0} = g(\sigma).
\]
Let $V^+ \subset V$ denote the sum of the nontrivial eigenspaces of $g(\sigma)$. We assume that $V$ is cofinite relative to the splitting $V = V^{g(\sigma)} \oplus V^+$.

\begin{defn}\label{F.definition}
Let $V$, $\om$, $h$ be as above, let $u \in V$, and let $M$ be an irreducible $h$-stable positive energy $V$-module. The supertrace function on $M$ is the function defined by
\[
F_M(\tau, z | u) = \str_M u_0 e^{2\pi i z (h_0 - p/2)} q^{L_0 - c/24}, \quad \text{where $\tau \in \HH$ and $z \in \C$},
\]
whenever the series on the right hand side converges.
\end{defn}
We remark that in general the conformal weight $\D(u)$ of a vector $u \in V$ and its conformal modes $u_n$ depend on the conformal structure $\om(\sigma)$. However for $u \in V^{\text{ne}}$ the conformal weight and modes coincide with those defined relative to $\om$, while for $u \in V^{(\alpha)}$, for $\alpha \neq 0$, the supertrace functions $F_M(\tau, z | u)$ vanishes. Therefore, although multiple conformal structures figure in the arguments of this section, we freely assume $u \in V^{\text{ne}}$, and we write $\D(u)$ and $u_n$ without risk of confusion.

In this section we apply Theorem \ref{JVEtheorem} to prove convergence of the series defining $F_M$ for $h$-stable irreducible $M$, and to compute the modular transformations of $F_M$. To make the connection with Theorem \ref{JVEtheorem} we introduce certain twisted supertrace functions $G_M$.
\begin{defn}
Let $V$, $\om$, $h$, $u$, and $M$ be as in Definition \ref{F.definition}. For all $\ell \in \Q$, and all $\sigma \in \Q_{>0}$ sufficiently small that $(V, \om(\sigma))$ be conformal and rational, we define
\[
G_M(\tau, \ell, \sigma | u) = \str_M u_0 g(\sigma)^{-\ell} q^{L_0(\sigma) - c(\sigma) / 24}.
\]
\end{defn}
\begin{lemma}\label{lemma.tracetoKW}
The functions $F_M$ and $G_M$ are related in the following way:
\begin{align}\label{tracetoKW}
G_M( \tau, \ell, \sigma | u ) = e^{2\pi i \frac{\sigma p \ell}{4}} q^{\frac{\sigma^2}{4}} F_M\left( \tau, \tfrac{\sigma}{2} (\ell + \tau) | u \right).
\end{align}
\end{lemma}

\begin{proof}
The proof is an easy computation. We have
\[
L_0(\sigma) - \frac{c(\sigma)}{24} = L_0 - \frac{c}{24} + \frac{\sigma}{2}h_0 - \frac{\sigma(p-\sigma)}{4}.
\]
Hence
\begin{align*}
G_M(\tau, \ell, \sigma | u)
&= \str_M u_0 e^{2\pi i \ell \frac{\sigma}{2} h_0} q^{L_0(\sigma) - c(\sigma)/24} \\
&= q^{\sigma^2/4} \str_M u_0 e^{2\pi i \ell \frac{\sigma}{2} h_0} q^{\frac{\sigma}{2}(h_0-p/2)} q^{L_0 - c/24} \\
&= q^{\sigma^2/4} \str_M u_0 \exp{ 2\pi i \tfrac{\sigma}{2} \left[ (\ell + \tau) h_0 - p \tau / 2 \right] } q^{L_0 - c/24} \\
&= q^{\sigma^2/4} \str_M u_0 \exp{ 2\pi i \tfrac{\sigma}{2} \left[ (\ell + \tau) (h_0 - p/2) + p \ell / 2 \right] } q^{L_0 - c/24} \\
&= e^{2\pi i \frac{\sigma p \ell}{4}} q^{\sigma^2/4} F_M\left(\tau, \tfrac{\sigma}{2}(\ell+\tau) | u \right).
\end{align*}
\end{proof}

The conditions we have imposed allow us to interpret $G_M$ (hence $F_M$) as an element of the space of conformal blocks $\CC(1, \ell)$ associated with the conformal vertex algebra $(V, \om(\sigma))$ and the automorphism $g_1 = g(\sigma)$. This can be used to prove convergence of the series defining $F_M$.
{
\begin{prop}\label{prop:convergence.of.F}
Let $M$ be a $h$-stable irreducible positive energy $V$-module. Then there exists $\varepsilon > 0$ such that the supertrace function $F_M(\tau, z | u)$ converges absolutely uniformly on compact subsets of the domain
\[
\{(\tau, z) \in \HH \times \C | 0 < \Imm(z) < \varepsilon \Imm(\tau)\},
\]
for each $u \in V$.
\end{prop}

\begin{proof}
By assumption $M$ is positive energy as a module over $(V, \omega(\sigma))$ for some $\sigma \in \Q_{>0}$, and this vertex algebra is rational. We recall that the graded pieces of an irreducible positive energy module over a rational vertex algebra are finite dimensional. It follows from this, together with our assumptions on $h$, that $M$ possesses an eigenspace decomposition $M = \bigoplus_{(n, \mu)} M_n^{(\mu)}$ where $L_0|_{M_n^{(\mu)}} = s+n$ and $h_0|_{M_n^{(\mu)}} = t+\mu$, the index $(n, \mu)$ ranges over $\Z_+ \times (\tfrac{1}{K}\Q)$ for some $K \in \Z_{>0}$ and $s$ and $t$ are two complex numbers. Moreover $\dim(M_{n}^{(\mu)}) < \infty$ for all $(n, \mu)$ and $\dim(M_{n}^{(\mu)}) = 0$ for any fixed $n$ and $\mu$ sufficiently positive. In fact the $h$-stable condition implies something slightly stronger. Let us introduce the support
\[
\supp(M) = \{(n, \mu) \in \Z_+ \times (\tfrac{1}{K}\Q) | \dim(M_{n}^{(\mu)}) > 0\}
\]
of $M$. Then there exist real numbers $\varepsilon > 0$ and $A$ such that $\supp(M) \subset \{(n, \mu) \in \Z_+ \times (\tfrac{1}{K}\Q) | \mu < \varepsilon n + A\}$. The condition that $\sigma$ be sufficiently small now becomes $0 < \sigma < \varepsilon$.

For $\sigma \in \Q$ satisfying $0 < \sigma < \varepsilon$ the twisted supertrace function $G_M$ is an element of the space of conformal blocks $\CC(1, \ell)$ associated with the conformal vertex algebra $(V, \om(\sigma))$ and the automorphism $g_1 = g(\sigma)$. In particular Theorem \ref{JVEtheorem} implies that the series defining $G_M$ converges absolutely. Therefore by Lemma \ref{lemma.tracetoKW} the series
\begin{align}\label{eq:biseries.for.F}
\sum_{(n, \mu) \in \supp(M)} q^{s + n - c/24} e^{2\pi i z (\mu_0 + \mu - p/2)} \left( \str_{M_n^{(\mu)}} u_0 \right),
\end{align}
which represents the supertrace function $F_M(\tau, z | u)$, also converges absolutely at $z = \sigma \tau / 2$.

Now we fix two values $0 < \sigma_1 < \sigma_2 < \varepsilon$, and we consider the sum (\ref{eq:biseries.for.F}) on the strip $\Imm(\sigma_1 \tau / 2) < \Imm(z) < \Imm(\sigma_2 \tau / 2)$. Let $\supp(M)_+$ and $\supp(M)_-$ denote the subsets of $\supp(M)$ in which $\mu \geq 0$, respectively $\mu < 0$, and denote by $F^{(\pm)}_M(\tau, z | u)$ the restriction of the sum (\ref{eq:biseries.for.F}) to $\supp(M)_\pm$. We observe that $F^{(+)}_M(\tau, z | u)$ converges since it is dominated by $F^{(+)}_M(\tau, \sigma_1 \tau / 2 | u)$, and $F^{(-)}_M(\tau, z | u)$ converges since it is dominated by $F^{(-)}_M(\tau, \sigma_2 \tau / 2 | u)$. Therefore $F_M$ converges absolutely uniformly on compact subsets of the strip. We now take $\sigma_2$ arbitrarily small and obtain the desired result.
\end{proof}
}

\begin{rem}
Convergence may be established more directly for $C_2$-cofinite $V$ as in {\cite[Appendix A]{HVE14}}. The supertrace functions are shown to satisfy differential equations whose coeffcients lie in a Noetherian ring of quasi-Jacobi forms. As in \cite{Z96} the Noetherian property implies convergence of the supertrace functions.
\end{rem}

Let $A = \twobytwo abcd \in SL_2(\Z)$. Theorem \ref{JVEtheorem} asserts that
\[
[G_M \cdot A](\tau, \ell, \sigma | u) = G_M\left( \frac{a\tau+b}{c\tau+d}, \ell, \sigma \bigg| (c\tau+d)^{-L_{[0]}(\sigma)} u\right)
\]
lies in $\CC(a+c\ell, b+d\ell)$, and is thus a linear combination of the $g(\sigma)^{b+d\ell}$-twisted supertrace functions
\begin{align}\label{aclbdl}
\str_{M'} u_0 \xi q^{L_0(\sigma) - c(\sigma) / 24}
\end{align}
on $g(\sigma)^{a+c\ell}$-twisted positive energy $V$-modules $M'$. Here $\xi$ is the automorphism defined in Subsection \ref{tr.fct.subsection}.

We now use Li's shift operator to reinterpret (\ref{aclbdl}) as a supertrace function on a $g(\sigma)$-twisted $V$-module. Indeed by Lemma \ref{twist=shift} there exists a $g(\sigma)$-twisted $V$-module $M^0$ such that $M' = M^0$ as vector superspaces, and
\[
Y^M(u, z) = Y^{M^0}(\D((a+c\ell-1)\tfrac{\sigma}{2}h, z)u, z).
\]
Moreover, under this identification, we have $\xi = g(\sigma)^{-(b+d\ell)}$. Hence (\ref{aclbdl}) is equal to
\begin{align}\label{pre-reduction}
\str_{M^0} \widehat{u}_0 e^{2\pi i (b + d\ell) \frac{\sigma}{2} h_0} q^{\widehat{L(\sigma)}_0 - c(\sigma)/24} \quad \text{where} \quad \widehat{u}_0 = [\D((a + c\ell - 1)\tfrac{\sigma}{2}h, 1)u]_0.
\end{align}
We apply Lemma \ref{lemma.shift.of.shift} to reduce (\ref{pre-reduction}) to
\begin{align}\label{post-reduction}
\str_{M^0} \widehat{u}_0 e^{2\pi i (b + d\ell) \frac{\sigma}{2} h_0} q^{L([a+c\ell]\sigma)_0 - c([a+c\ell]\sigma)/24}.
\end{align}
Using formula (\ref{zeromode.shift}) we may write (\ref{post-reduction}) as
\begin{align*}
& \str_{M^0} [\D((a+c\ell-1)\tfrac{\sigma}{2}h, 1)u]_0 g((a+c\ell)\sigma)^{-\frac{b+d\ell}{a+c\ell}} q^{L_0((a+c\ell)\sigma) - c((a+c\ell)\sigma)/24} \\
= {} & G_{M^0}\left( \tau, \frac{b+d\ell}{a+c\ell}, (a+c\ell)\sigma \bigg| \D((a + c\ell - 1)\frac{\sigma}{2}h, 1) u \right).
\end{align*}
\begin{rem}
Shift operators do not in general preserve the positive energy condition. However rationality relative to $h$ guarantees that $\sigma$ may be chosen sufficiently small that both $(V, \om(\sigma))$ and $(V, \om((a + c\ell)\sigma))$ are rational conformal vertex algebras, all of whose irreducible modules are positive energy.
\end{rem}
The outcome of the preceding discussion is the following proposition.
\begin{prop}\label{Gmodular}
Let $V$, $\om$, $h$ and $u$ be as in Definition \ref{F.definition}. Fix $\ell \in \Z_{>0}$ and $A = \twobytwo abcd \in SL_2(\Z)$. Then there exists a matrix $\ov{\rho} = \ov{\rho}_{M, M'}$, whose entries depend on $\ell, \sigma$, such that for each irreducible $h$-stable positive energy $V$-module $M$ the relation
\begin{align*}
G_{M}\left( \frac{a\tau+b}{c\tau+d}, \ell, \sigma \bigg| (c\tau+d)^{-L_{[0]}(\sigma)} u \right) = \sum_{M'} \ov{\rho}_{M, M'} G_{M'}\left( \tau, \frac{b+d\ell}{a+c\ell}, (a+c\ell)\sigma \bigg| \D((a + c\ell - 1)\frac{\sigma}{2}h, 1)u \right)
\end{align*}
holds \textup{(}where the sum runs over all irreducible $h$-stable positive energy $V$-modules $M'$\textup{)}.
\end{prop}

We now convert the modular transformations for $G_M$ into modular transformations for $F_M$. We begin by temporarily restricting attention to $u = \vac$ (and omitting it from the notation). By Lemma \ref{lemma.tracetoKW} we have
\begin{align*}
G_M\left( \frac{a\tau+d}{c\tau+d}, \ell, \sigma \right) &= \exp{2\pi i \frac{p\ell}{4}} \exp{ 2\pi i \left[ \frac{\sigma^2}{4} \cdot \frac{a\tau+b}{c\tau+d} \right] } F_M\left( \frac{a\tau+d}{c\tau+d}, \frac{\sigma}{2} \left[ \ell + \frac{a\tau+b}{c\tau+d} \right] \right) \\
\text{and} \quad
G_M\left( \tau, \frac{b+d\ell}{a+c\ell}, (a+c\ell)\sigma \right) &= \exp{2\pi i \frac{p(a+c\ell)}{4}} \exp{ 2\pi i \left[ \frac{\sigma^2}{4} (a+c\ell)^2 \tau \right] } F_M\left( \tau, \frac{\sigma}{2} \left[ (b+d\ell) + (a+c\ell)\tau \right] \right).
\end{align*}
So Proposition \ref{Gmodular} implies
\begin{align}\label{Gmodular.itoF}
\begin{split}
F_M\left( \frac{a\tau+d}{c\tau+d}, \frac{\sigma}{2} \left[ \ell + \frac{a\tau+b}{c\tau+d} \right] \right)
= {} & \exp{2\pi i \frac{p(a+c\ell-1)}{4}} \exp{ 2\pi i \frac{\sigma^2}{4} \left[ (a+c\ell)^2 \tau - \frac{a\tau+b}{c\tau+d} \right] } \times \\
& \sum_{M'} \ov{\rho}_{M, M'} F_{M'}\left( \tau, \frac{\sigma}{2} \left[ (b+d\ell) + (a+c\ell)\tau \right] \right).
\end{split}
\end{align}
We make the substitution
\begin{align}\label{sigma.to.z}
z = \frac{\sigma}{2} \left[ (b+d\ell) + (a+c\ell)\tau \right],
\end{align}
and calculate
\begin{align}\label{subst.consequences}
\begin{split}
\frac{\sigma}{2} \left[ \ell + \frac{a\tau+b}{c\tau+d} \right]
&= \frac{z}{c\tau+d} \\
\frac{\sigma^2}{4} \left[ (a+c\ell)^2 \tau - \frac{a\tau+b}{c\tau+d} \right]
&= \frac{c z^2}{c\tau+d} - \frac{\sigma^2}{4}(ab + 2bc \ell + cd \ell^2).
\end{split}
\end{align}

Thus (\ref{Gmodular.itoF}) becomes
\begin{align*}
F_M\left( \frac{a\tau+d}{c\tau+d}, \frac{z}{c\tau+d} \right)
= \exp{\left[ 2\pi i \frac{c z^2}{c\tau+d} \right]} \sum_{M'} {\rho}_{M, M'} F_{M'}\left( \tau, z \right),
\end{align*}
where
\begin{align}\label{relationbetweenSmatrices}
\rho_{M, M'} = \exp \frac{2\pi i}{4} \left[{ \sigma p(a+c\ell-1) - \sigma^2 (ab + 2bc \ell + cd \ell^2) }\right] \times \ov{\rho}_{M, M'}.
\end{align}

To determine the modular behaviour of the functions $F_M(\tau, z | u)$ for general $u \in V$, we require the following Baker-Campbell-Hausdorff type formulas.
\begin{lemma}\label{BCHforus}
Let $X, Y$ be locally finite even operators on a vector superspace $U$, satisfying $[X, Y] = s Y$ for some constant $s$. We have
\begin{align*}
\exp(\al X) Y \exp(-\al X) = e^{\al s} Y.
\end{align*}
If, moreover, $s \neq 0$ then for any constants $\al$ and $\beta$ we have
\begin{align*}
\exp(\al X) \exp(\beta Y) = \exp\left( \al X + \frac{\al s}{1-e^{-\al s}} \beta Y \right)
\end{align*}
and
\[
\exp(\al X) \exp({\beta Y}) \exp(-\al X) = \exp({\beta e^{\al s} Y}).
\]
\end{lemma}

\begin{prop}\label{BCHforSL2}
Let $X$ and $Y$ be locally finite even operators on a vector superspace $U$, satisfying
\[
[X, Y] = Y.
\]
On the space of functions on $(\tau, z, u) \in \HH \times \C \times U$, linear in $u$, the formula
\begin{align*}
[\varphi \cdot A](\tau, z | u) = \varphi\left( \frac{a\tau+b}{c\tau+d}, \frac{z}{c\tau+d} \bigg| (c\tau+d)^{X} e^{- \frac{c z}{c\tau+d} Y} u \right), \quad \text{where} \quad A = \left(\begin{array}{cc}a & b \\ c & d \\ \end{array}\right),
\end{align*}
defines a right action of $SL_2(\Z)$.
\end{prop}


\begin{proof}
Let $A = \twobytwo abcd$ and $B = \twobytwo{a'}{b'}{c'}{d'}$. Write $A\tau = \frac{a\tau+b}{c\tau+b}$, and $\ga_A(\tau) = (c\tau+d)^{-1}$. It is well known that
\begin{align}\label{SL2cocycles}
\begin{split}
\ga_A(B\tau)\ga_B(\tau) &= \ga_{AB}(\tau), \\
\text{and} \quad 
c'\ga_B(\tau) + c \ga_B(\tau)^2 \ga_A(B\tau) &= (ca'+dc')\ga_{AB}(\tau).
\end{split}
\end{align}
Now we calculate
\begin{align}\label{SL2.action.check}
[(\varphi A) B](\tau, z | u)
&= [\varphi A]\left( B\tau, \ga_B(\tau) z \bigg| \ga_B(\tau)^{-X} e^{ c' \ga_B(\tau) z Y} u \right) \nonumber \\
&= \varphi\left( A(B\tau), \ga_A(B\tau) \ga_B(\tau) z \bigg| \ga_A(B\tau)^{-X} e^{c \ga_A(B\tau) \ga_B(\tau) z Y} \ga_B(\tau)^{-X} e^{ c' \ga_B(\tau) z Y} u \right).
\end{align}
The third formula of Lemma \ref{BCHforus}, with
\[
\al = \log{\ga_B(\tau)} \quad \text{and} \quad \beta = c \ga_A(B\tau) \ga_B(\tau) z,
\]
implies
\begin{align*}
e^{c \ga_A(B\tau) \ga_B(\tau) z Y} \ga_B(\tau)^{-X} = \ga_B(\tau)^{-X} e^{c \ga_A(B\tau) \ga_B(\tau)^2 z Y}.
\end{align*}
Substituting this into (\ref{SL2.action.check}) and using (\ref{SL2cocycles}) reduces $[(\varphi A) B](\tau, z | u)$ to
\[
\varphi\left( (AB)\tau, \ga_{AB}(\tau)z | \ga_{AB}(\tau)^{-X} e^{(ca'+dc')\ga_{AB}(\tau) z Y} u \right) = [\varphi(AB)](\tau, z | u).
\]
\end{proof}

The following proposition and the subsequent theorem, are the main results of this section.
{
\begin{prop}\label{prop.Fmodular}
Let $V$, $\om$, $h$, and $u$ be as in Definition \ref{F.definition}. Let $A = \twobytwo abcd$ and $\ell \in \Z_{>0}$ and set $z = (\sigma/2)[(b+d\ell) + (a+c\ell)\tau]$. Then for all $\tau \in \HH$ and $\sigma \in \Q_{>0}$ sufficiently small,
\begin{align}\label{Fmodular}
F_{M}\left( \frac{a\tau+b}{c\tau+d}, \frac{z}{c\tau+d} \bigg| (c\tau+d)^{-L_{[0]}} \exp{\left[ -\frac{c z}{c\tau+d} {I(h)} \right]} u \right) = \exp\left[ 2\pi i \frac{c z^2}{c\tau+d} \right] \sum_{M'} \rho_{M, M'} F_{M'}(\tau, z | u),
\end{align}
where $\rho_{M, M'}$ is the matrix of \textup{(}\ref{relationbetweenSmatrices}\textup{)}.
\end{prop}

}

\begin{proof}
We start by proving the formula 
\begin{align}\label{quasi-part}
\begin{split}
& (c\tau+d)^{-L_{[0]}(\sigma)} \D({-(a + c\ell - 1)\frac{\sigma}{2}h}, 1) u \\
& \phantom{------} = (c\tau+d)^{-L_{[0]}} \exp{ \left[ -\frac{c}{c\tau+d} \cdot \frac{\sigma}{2} [(b+d\ell) + (a+c\ell)\tau] { I(h)} \right]} u.
\end{split}
\end{align}
Note that
\begin{align*}
[(2\pi i)^2\om_{([1])}, (2\pi i)^2h_{([1])}] = -(2\pi i)^2 h_{([1])},
\end{align*}
and that
\begin{align*}
\begin{split}
L_{[0]}(\sigma)
= (2\pi i)^2 \left[ \om_{([1])} - \frac{\sigma}{2} (Th)_{([1])} \right]
= h_0 + (2\pi i)^2 \left[ \om_{([1])} + \frac{\sigma}{2} h_{([1])} \right].
\end{split}
\end{align*}
We plug $X = -(2\pi i)^2 \om_{([1])}$ and $Y = \frac{\sigma}{2} I(h) = \frac{\sigma}{2} (2\pi i)^2 h_{([1])}$, along with the parameters
\[
s = 1, \quad \al = \log(c\tau+d), \quad \text{and} \quad \beta = \frac{1}{c\tau+d} - 1,
\]
into Lemma \ref{BCHforus}, to obtain
\begin{align*}
(c\tau+d)^{-L_{[0]}(\sigma)}
&= (c\tau+d)^{- h_0 - (2\pi i)^2 \om_{([1])} - \frac{\sigma}{2} (2\pi i)^2 h_{([1])}} \\
&= (c\tau+d)^{ - h_0 -(2\pi i)^2 \om_{([1])} } \exp \left[ \left( \frac{1}{c\tau+d} - 1 \right) \frac{\sigma}{2} (2\pi i)^2 h_{([1])} \right].
\end{align*}
This, together with (\ref{compareZhuLi}), yields
\begin{align}\label{quasi-part.2}
\begin{split}
& (c\tau+d)^{-L_{[0]}(\sigma)} \D\left({-(a + c\ell - 1)\frac{\sigma}{2}h}, 1\right) u \\
& \phantom{------} =
(c\tau+d)^{-L_{[0]}} \exp \left[ \left( \frac{1}{c\tau+d} - 1 \right) \frac{\sigma}{2} (2\pi i)^2 h_{([1])} - \frac{\sigma}{2} (a+c\ell-1) (2\pi i)^2 h_{([1])} \right].
\end{split}
\end{align}
The term in square brackets here is $\frac{\sigma}{2} (2\pi i)^2 h_{([1])}$ times
\begin{align*}
\begin{split}
\frac{1}{c\tau+d} - (a + c\ell)
&= \frac{1 - (a+c\ell)(c\tau+d)}{c\tau+d} \\
&= \frac{1 - c(a+c\ell)\tau - ad - cd\ell}{c\tau+d} \\
&= -\frac{c(a+c\ell)\tau + bc + cd\ell}{c\tau+d}
= -\frac{c}{c\tau+d} \left[ (a+c\ell)\tau + b + d\ell \right].
\end{split}
\end{align*}
Substituting this into (\ref{quasi-part.2}) yields (\ref{quasi-part}).

Now we substitute $\D({-(a + c\ell - 1)\frac{\sigma}{2}h}, 1)u$ in place of $u$ in Proposition \ref{Gmodular} to obtain
\begin{align*}
G_M\left( \frac{a\tau+b}{c\tau+d}, \ell, \sigma \bigg| (c\tau+d)^{-L_{[0]}(\sigma)} \D({-(a + c\ell - 1)\frac{\sigma}{2}h}, 1) u \right)
= \sum_{M'} \ov{\rho}_{M, M'} G_{M'}\left( \tau, \frac{b+d\ell}{a+c\ell}, (a+c\ell)\sigma \bigg| u \right).
\end{align*}
Substituting (\ref{quasi-part}) transforms this into
\begin{align*}
& G_M\left( \frac{a\tau+b}{c\tau+d}, \ell, \sigma \bigg| (c\tau+d)^{-L_{[0]}} \exp{ \left[ -\frac{c}{c\tau+d} \cdot \frac{\sigma}{2} [(b+d\ell) + (a+c\ell)\tau] {(2\pi i)^2 h_{([1])}} \right]} u \right) \\
& \phantom{------} = \sum_{M'} \ov{\rho}_{M, M'} G_{M'}\left( \tau, \frac{b+d\ell}{a+c\ell}, (a+c\ell)\sigma \bigg| u \right).
\end{align*}

We make the substitution (\ref{sigma.to.z}) again, using (\ref{subst.consequences}) and (\ref{relationbetweenSmatrices}). Thus we obtain (\ref{Fmodular}).
\end{proof}

For the applications we wish to pursue, the most useful consequence of Proposition \ref{prop.Fmodular} is the following theorem.
\begin{thm}\label{theorem1.body}
Let $V$, $\om$, $h$, and $u$ be as in Definition \ref{F.definition}. Let $\mathcal{X}$ denote the set of $h$-stable irreducible positive energy $V$-modules, and let
\[
F_M(\tau, z | u) = \str_M u_0 e^{2\pi i z (h_0 - p/2)} q^{L_0 - c/24}
\]
be the supertrace function associated with $M \in \mathcal{X}$. Then there exists $\varepsilon > 0$ such that for each $M \in \mathcal{X}$ and all $u \in V$ the supertrace function $F_M(\tau, z | u)$ converges absolutely uniformly on compact subsets of the domain
\[
\{(\tau, z) \in \HH \times \C | 0 < \Imm(z) < \varepsilon \Imm(\tau)\}.
\]
Now we suppose \textup{(}1\textup{)} that the set of functions $F_M(\tau, z | \vac)$ as $M$ runs over $\mathcal{X}$ is modular invariant, i.e., that there exists a representation $\upsilon$ of $SL_2(\Z)$ for which (\ref{Fmodular.repeat}) holds for $u = \vac$, and \textup{(}2\textup{)} that for each $\alpha \in (0, \varepsilon) \subset \R$ and $\beta \in \R$ the set of functions $F_M(\tau, \alpha\tau + \beta | \vac)$, as $M$ runs over $\mathcal{X}$, is linearly independent. Then for all $u \in V$ the relation
\begin{align}\label{Fmodular.repeat}
F_{M}\left( \frac{a\tau+b}{c\tau+d}, \frac{z}{c\tau+d} \bigg| (c\tau+d)^{-L_{[0]}} \exp{\left[ -\frac{c z}{c\tau+d} {I(h)} \right]} u \right) = \exp\left[ 2\pi i \frac{c z^2}{c\tau+d} \right] \sum_{M'} \upsilon_{M, M'} F_{M'}(\tau, z | u)
\end{align}
is satisfied in the intersection of the domains of convergence of the two sides.
\end{thm}

\begin{proof}
Since $\mathcal{X}$ is finite the statement on convergence follows directly from Proposition \ref{prop:convergence.of.F}, taking $\varepsilon$ to be the minimum of the values that appear there. Next by combining equation (\ref{Fmodular}) at $u = \vac$ with relation (\ref{Fmodular.repeat}), we obtain
\begin{align}\label{equality.lin.comb}
\sum_{M'} \left( {\rho}_{M, M'} - {\upsilon}_{M, M'} \right) F_{M'}\left( \tau, \frac{\sigma}{2} \left[ (b+d\ell) + (a+c\ell)\tau \right] \right) = 0,
\end{align}
where $\rho_{M, M'}$ is given by (\ref{relationbetweenSmatrices}). Importantly $\rho_{M, M'} = {\rho}_{M, M'}(A, \ell, \sigma)$ is independent of $\tau$. By hypothesis, i.e., by condition (2), the functions that appear in (\ref{equality.lin.comb}) are linearly independent for any sufficiently small $\sigma$. Therefore ${\rho}_{M, M'} = {\upsilon}_{M, M'}$. Now Proposition \ref{prop.Fmodular} implies that for each fixed $\tau \in \HH$ the relation (\ref{Fmodular.repeat}) holds at $z = \frac{\sigma}{2} \left[ (b+d\ell) + (a+c\ell)\tau \right]$. By allowing $\sigma$ to vary over a set of the form $\Q \cap (0, \varepsilon_1)$ we obtain the identity (\ref{Fmodular.repeat}) on a set of values that has accumulation points within the intersection of the domains of convergence of the two sides of (\ref{Fmodular.repeat}). The desired conclusion now follows by the identity theorem for holomorphic functions.
\end{proof}

Finally we record the following corollary in the rational and $C_2$-cofinite case. It is related to a result independently obtained in \cite{Krauel.C2.case}.
\begin{cor}\label{C2.case}
Let $(V, \om)$ be a rational and $C_2$-cofinite conformal vertex algebra graded by integer conformal weights. Let $h \in V_1$ satisfy the OPE relations \textup{(}\ref{main.lambda.relation}\textup{)}, the grading condition stated thereafter, and the linear independence condition \textup{(}2\textup{)} of Theorem \ref{theorem1.body}. Then the supertrace functions $F_M(\tau, z | u)$ on the irreducible positive energy $V$-modules satisfy the relation \textup{(}\ref{Fmodular}\textup{)} where $\rho$ is some linear representation of the group $SL_2(\Z)$.
\end{cor}

\begin{proof}
The modular invariance of the restricted trace functions $F_M(\tau, z | \vac)$ has been established by Miyamoto \cite{Miyamoto.theta} for even vertex algebras. The extension to vertex algebras with nontrivial odd part is straightforward, the important condition to maintain is that the conformal weights be integers. The modular invariance of the $F_M$ at arbitrary $u \in V$ now follows from Theorem \ref{theorem1.body}.
\end{proof}

\section{Admissible Affine Vertex Algebras}\label{adm.aff.case}

Let $\ov{\g}$ be a simple Lie algebra, and let $k$ be an admissible number for $\ov\g$. We denote by
\begin{align*}
C^{\circ}_- = \{h \in \ov\h_\Q | \text{$\al(h) < 0$ for all $\al \in \ov\Pi$}\}.
\end{align*}
the negative open fundamental chamber of $\ov\g$.
\begin{lemma}\label{adm.rel.rational}
Let $h \in C^\circ_-$, then $V_k(\ov\g)$ is rational relative to $h$.
\end{lemma}

\begin{proof}
The condition $h \in C^\circ_-$ guarantees that $\om(\sigma)$ is a conformal vector for sufficiently small $\sigma \in \Q_{>0}$. The shift of conformal structure from $\om$ to $\om(\sigma)$ has the effect of destroying the positive energy condition for those $(V_k(\ov\g), \om)$-modules outside the category $\OO_k$. Indeed the positive energy irreducible $(V_k(\ov\g), \om(\sigma))$-modules are precisely those irreducible $V_k(\ov\g)$-modules which lie in $\OO_k$. Thus rationality of $V_k(\ov\g)$ relative to $h$ is equivalent to rationality of $V_k(\ov\g)$ in the category $\OO_k$, which is the main theorem of \cite{A12catO} (cf. Theorem \ref{rational.category.O.theorem} above).
\end{proof}

Recall that $V_k(\ov\g)$ is graded by the root lattice $\ov Q$.
\begin{lemma}\label{adm.rel.cofinite}
Let $h \in \ov{\h}_\Q$ and $\sigma \in \Q$ be such that $\frac{\sigma}{2} \al(h) \notin \Z$ for all roots $\al$ of $\ov{\g}$, and that the automorphism
\[
g(\sigma) = \exp\left(-2\pi i \tfrac{\sigma}{2} h_0\right)
\]
has prime order. Let $W$ be the $\ov Q$-graded complement in $V = V_k(\ov\g)$ to the fixed point subalgebra $V^{g(\sigma)}$. Then $V$ is cofinite relative to the splitting $V = V^{g} \oplus W$ for each nontrivial element $g \in \left<g(\sigma)\right>$.
\end{lemma}

\begin{proof}
By the condition that $g(\sigma)$ be of prime order, we have $V^{g}=V^{g(\sigma)}$ and so it suffices to verify the claim on $g = g(\sigma)$ itself. Let $C_2(V) = V_{(-2)}V$ and $C^\text{rel}(V) = V^g_{(-2)}V^g + V_{(-1)}W$. Since $V$ is strongly generated by $\ov\g = V_1$, we have a surjection of {commutative} algebras $S[{\ov\g}^*] \twoheadrightarrow V / C_2(V)$. The conditions imposed on $\sigma$ and $h$ ensure that each root space $\ov\g_\al$ lies outside $V^g$. It follows from a standard PBW argument that the quotient of $V$ by $C^\text{rel}(V)$ is naturally a quotient of $U(\ov\h)$. Thus we have a commutative diagram of surjections of {commutative} algebras
\begin{align*}
\xymatrix{
S[\ov\g^*] \ar@{->>}[r] \ar@{->>}[d] & R(V) \ar@{->>}[d] \\
U(\ov\h) \ar@{->>}[r] & R^\text{rel}(V). \\
}
\end{align*}

We recall the \emph{associated variety} of $V$, which is the affine scheme $X_V = \spec{R(V)}$ \cite{Ara12}. We also put $X^\text{rel}_V = \spec{R^\text{rel}(V)}$ the \emph{relative associated variety}. It was shown in {\cite[Theorem 5.3.1]{A.assoc.var}} that if $V = V_k(\ov\g)$ where $k \in \Q$ is an admissible number for $\ov\g$ then $X_V \subset \mathcal{N}$ where $\mathcal{N} \subset \ov\g$ is the nilpotent cone.

Our diagram above implies that $X^\text{rel}_V \subset X_V \cap \ov\h \subset \mathcal{N} \cap \ov\h = \{0\}$. It follows that $\dim_\C{R^\text{rel}(V)} < \infty$, and $V = V_k(\ov\g)$ is relatively cofinite as required.
\end{proof}

Recall the set $Y = \HH \times \ov\h$ and its subset $Y^+$ defined in (\ref{Yplus.def}), as well as the hyperplanes $H_{\al, \om}$. We now introduce the trace function
\[
\Psi_M(\tau, x | u) = \tr_{M} u_0 e^{2\pi i x_0} q^{L_0 - c_k/24}
\]
of $(\tau, x) \in Y$ and $u \in V^k(\ov\g)$ on the highest weight $\g$-module $M$. We also denote $\Psi_\la = \Psi_{L(\la)}$ for $\la$ admissible. The $\Psi_M$ specialise at $u=\vac$ to the Kac-Wakimoto characters $\chi_M$ of (\ref{KWchar.gen.def}). We have seen in Section \ref{main.computation} that, on general grounds, $\Psi_\la(\tau, zh | u)$ converges on sets of the form $0 < \textup{Im}({z}) < \varepsilon \textup{Im}({\tau})$. The following lemma gives more precise information on convergence, but is not necessary for what follows and therefore can be skipped.
\begin{lemma}\label{tr.fct.convergence}
Let $M$ be a highest weight $\g$-module, which we regard as a $V^k(\ov\g)$-module. For any $u \in V^k(\ov{\g})$, the series defining $\Psi_M(\tau, x | u)$ converges absolutely to a holomorphic function on $Y^+$, and extends to a meromorphic function on $Y$ with possible poles on the hyperplanes $H_{\al, \om}$ for $\al \in \ov{\D}_+^\vee$, $\om \in \Z + \Z \tau$.
\end{lemma}

\begin{proof}
There is a grading $V = V^k(\ov\g) = \bigoplus_{\al \in \ov Q} V_\al$ by the root lattice $\ov Q$, extending that on $\ov\g$. For any $u \in V_\al$ we have the commutation relations $[x_0, u_{m}] = \al(x) u_m$ and $[L_0, u_m] = -m u_m$. Hence
\begin{align}\label{comm.Val}
e^{2\pi i x_0} u_m = e^{2\pi i \al(x)} u_m e^{2\pi i x_0} \quad \text{and} \quad q^{L_0} u_m = q^{-m} u_m q^{L_0}.
\end{align}

Consider the increasing and exhaustive filtration $L^\bullet V$ defined by
\[
L^0V = \C \vac, \quad L^p V = L^{p-1}V + \sum_{n \leq -1} \ov\g_{(n)}L^{p-1}V.
\]
We show the sum defining $\Psi_M(\tau, x | u)$ converges on $Y^+$ for $u \in L^p V$ by induction on $p$. The base case of $p=0$ is the convergence of $\chi_M$.

Suppose the claim is proved for $L^{p-1}V$, and for all elements of $L^p V$ of conformal weight less than $\D$. Let $u \in L^p V$ be of conformal weight $\D$. Either $u \in L^{p-1}V$ or $u = a_{(n)}b$ for some $n \leq -1$, $a \in \ov\g$, and $b \in L^{p-1}V$. By the PBW theorem we may assume without loss of generality that $a \in \ov\h$ or $a \in \ov\g_\al$ for some $\al \in \ov\D_+$. We write $\ov\g_0 = \ov\h$ for convenience.

Borcherds identity implies
\begin{align}\label{implication.borcherds}
\sum_{i \in \Z_+} \binom{m+\D(a)-1}{i} (a_{(n+i)}b)_{0} = \sum_{i \geq 0} (-1)^i \binom{n}{i} \left[ a_{n-i} b_{-n+i} - (-1)^n b_{-i} a_{i} \right].
\end{align}
By the inductive assumption $\Psi_M(\tau, x | a_{(n+i)}b)$ converges on $Y^+$ for $ i > 0$. So it suffices to analyse the trace of the right hand side.

Let $j \in \Z_+$. Using the commutation relations (\ref{comm.Val}) and the symmetry of the trace we obtain
\begin{align*}
\tr_{M} a_{-j} b_j e^{2\pi i x_0} q^{L_0}
&= q^j e^{2\pi i \al(x)} \tr_{M} b_j a_{-j} e^{2\pi i x_0} q^{L_0}.
\end{align*}
If $j > 0$ then $b_j$ is locally nilpotent, and we may deduce
\begin{align*}
\tr_{M} a_{-j} b_j e^{2\pi i x_0} q^{L_0} = \frac{q^j e^{2\pi i \al(x)}}{1-q^j e^{2\pi i \al(x)}} \sum_{i \in \Z_+} \binom{j}{i} \tr_M (b_{(i)}a)_0 e^{2\pi i x_0} q^{L_0}.
\end{align*}
If $a \in \ov\g_{\al}$ for $\al > 0$ then the same reasoning applies. 

The trace $\tr_M \left( \cdots \right) e^{2\pi i x_0} q^{L_0}$ of the right hand side of (\ref{implication.borcherds}) is thus reduced to
\begin{align*}
\sum_{i \in \Z_+} \tr_M (a_{(i)}b)_0 e^{2\pi i x_0} q^{L_0} \left[ \sum_{j \geq 0} (-1)^j \binom{n}{j} \left( \binom{-n+j}{i} \frac{q^{-n+j} e^{2\pi i \al(x)}}{1-q^{-n+j} e^{2\pi i \al(x)}} - (-1)^n \binom{j}{i} \frac{q^{j} e^{-2\pi i \al(x)}}{1-q^{j} e^{-2\pi i \al(x)}} \right) \right].
\end{align*}
The denominators appearing within the $j$-summation are uniformly bounded on compact subsets of $Y^+$, and the summations have radius of convergence $1$ in $q$ by the ratio test. The $i$-summation is finite and by our inductive assumption each function $\tr_M (a_{(i)}b)_0 e^{2\pi i x_0} q^{L_0}$ is convergent on $Y^+$. The convergence of $\Psi_M(\tau, x | u)$ follows by induction.

The case of $\al = 0$ must be handled separately, but follows easily from the identity
\[
\tr_M b_0 a_0 e^{2\pi i x_0} q^{L_0} = \left. \frac{d}{dt} \tr_M b_0 e^{2\pi i (x + t a)_0} q^{L_0} \right|_{t=0}.
\]

It is also clear from the induction that multiplication of $\Psi_M(\tau, x | u)$ by a sufficiently high power of
\[
\prod_{j \in \Z_+} (1-q^{n+1})^\ell \cdot \prod_{\al \in \D_+} \prod_{j \in \Z_+} (1 - q^{n} e^{-2\pi i \al(x)}) (1 - q^{n+1} e^{2\pi i \al(x)}),
\]
renders it expressible by a series which, for any fixed $|q| < 1$, has infinite radius of convergence in $x \in \ov\h$. The meromorphicity statement follows.
\end{proof}

\begin{lemma}\label{lemma:admiss.lin.indep}
Let $\ov{\g}$ be a simple Lie algebra and $k$ an admissible level. There exists a nonempty Zariski open subset $D \subset C^\circ_-$ (a complement of finitely many hyperplanes) such that for any $h \in D$ the specialised characters $\chi_{\la}(\tau, \tau h)$, as $\la$ runs over $\coorprin^k$, are linearly independent functions of $\tau$.
\end{lemma}

\begin{proof}
By the definition of $C^\circ_-$ we know that $\chi_{\la}(\tau, \tau h)$ is the sum of an absolutely convergent series in powers of $q$. From the weight space decomposition of $L(\la)$ it follows immediately that the leading term of this series is $q^{\la(h) + h_\la - c_k/24}$. The equality of any pair of the exponents $\la(h) + h_\la - c_k/24$ defines a certain hyperplane in $\ov\h$. We define $D$ to be the complement in $C^\circ_-$ of these hyperplanes, and take $h \in D$ now. Since the power series defining the $\chi_{\la}(\tau, \tau h)$ all begin with unequal powers of $q$, they are clearly linearly independent functions.
\end{proof}

Now we are ready to prove the main theorem of this section.
\begin{thm}\label{theorem1.affine}
Let $\ov\g$ be a simple Lie algebra, and $k \in \Q$ a \textup{(}co\textup{)}principal admissible number for $\ov\g$. Then for all $\la \in \coorprin^k$ we have
\begin{align*}
\Psi_\la\left( \frac{a\tau+b}{c\tau+d}, \frac{x}{c\tau+d} \bigg| (c\tau+d)^{-L_{[0]}} \exp{\left[ -\frac{c}{c\tau+d} {I(x)} \right]} u \right) = \exp\left[ \pi i k \frac{c (x, x)}{c\tau+d} \right] \sum_{\la' \in \coorprin^k} \rho_{\la, \la'}(A) \Psi_{\la'}(\tau, x | u).
\end{align*}
The $S$-matrix $a(\la, \la') = \rho_{\la, \la'}\twobytwo{0}{-1}{1}{0}$ is given by Theorem \ref{KWSmatrix} \textup{(}resp.\ Theorem \ref{KWSmatrix.coprin}\textup{)} in the case that $k$ be principal \textup{(}resp.\ coprincipal\textup{)}.
\end{thm}

\begin{proof}
Let $h \in C^\circ_-$, such that $\left<h, h\right> = 2$. 
Relation (\ref{main.lambda.relation}) holds with $p=0$ and the specialisation $\Psi_\la(\tau, zh | u)$ coincides with the trace function $F_{L(\la)}(\tau, z | u)$ of Definition \ref{F.definition}. By Lemmas \ref{adm.rel.rational} and \ref{adm.rel.cofinite} we have rationality and cofiniteness of $V$ relative to $h$ (we observe that the arguments of Section \ref{main.computation} are unaffected by the prime order restriction on $\sigma$ coming from Lemma \ref{adm.rel.cofinite}). We wish to verify the conditions (1) and (2) of Theorem \ref{theorem1.body} for the functions $\chi_\la(\tau, z) = \Psi_\la(\tau, z | \vac)$. Condition (1), modular invariance, is guaranteed by Theorems \ref{KWSmatrix} and \ref{KWSmatrix.coprin}. Condition (2) follows from Lemma \ref{lemma:admiss.lin.indep}. Indeed the functions
\begin{align*}
F_{L(\la)}(\tau, \al\tau+\beta)
= \chi_\la(\tau, (\alpha\tau+\beta)h)
= \tr_{L(\la)} e^{2\pi i \beta h_0} q^{L_0 + \alpha h_0 - c_k/24}
\end{align*}
are linearly independent whenever the set of numbers $h_\la + \alpha \la(h)$, as $\la$ ranges over $\coorprin^k$, are distinct.

We are now able to apply Theorem \ref{theorem1.body} to obtain (\ref{Fmodular.repeat}) with $\mathcal{X} = \coorprin^k$, $F_{L(\la)} = \Psi_\la$ and $\upsilon_{L(\la), L(\la')} = \rho_{\la, \la'}$ is the Kac-Wakimoto representation. Substituting for $x = z h$, and using $\left<h, h\right> = k (h, h)$, in (\ref{Fmodular}) yields the formula asserted in the theorem statement. Now since the set $D$ of Lemma \ref{lemma:admiss.lin.indep} contains a $\C$-linear basis of $\ov\h$ we deduce, by the identity theorem for holomorphic functions, that the modular transformation relation holds as an identity of functions of $(\tau, x) \in \HH \times \C$, valid on the intersection of the domains of convergence of the two sides.
\end{proof}

\section{The Charged Free Fermions}\label{ghost.section}

We recall the theta product
\begin{align}\label{theta.definition}
\Theta(\tau, z) = \frac{\theta_{11}(\tau, z)}{\eta(\tau)} = q^{1/12} e^{\pi i z} \prod_{n=1}^\infty{(1-e^{2\pi i z} q^{n-1})(1-e^{-2\pi i z} q^{n})},
\end{align}
and the classical modular relation {\cite[pp. 475]{Whit.Wat}}
\begin{align}\label{classical.theta}
\Theta(-1/\tau, z/\tau)
= -i e^{\pi i z^2/\tau} \Theta(\tau, z).
\end{align}

Let $U$ be a finite dimensional vector space. The Clifford Lie superalgebra $\cliff(U)$ is defined by 
\[
\cliff(U) = (U \oplus U^*)[t, t^{-1}] \oplus \C 1, \quad [a_m, b_n] = \left<a, b\right> 1
\]
where $U \oplus U^*$ is given odd parity and $\C 1$ even, $a_m$ denotes $at^m$, and $\left<,\right>$ is the natural symmetric pairing on $U \oplus U^*$ defined by $\left<\al, x\right> = \left<x, \al\right> = \al(x)$ for all $x \in U$, $\al \in U^*$.

For now we take $U = \C \psi$ one dimensional, and let $\bigwedge$ be the Fock $\cliff(\C\psi)$-module generated from the highest weight vector $\vac$, subject to the relations $\psi_n \vac = 0$ for $n > 0$, and $\psi^*_n \vac = 0$ for $n \geq 0$. The module $\bigwedge$ has a vertex algebra structure {\cite[Section 3.6]{KacVA}} (which goes by several names, including charged free fermions, and the ghost system) with generating fields
\[
\psi(z) = \sum_n \psi_n z^{-n} \quad \text{and} \quad \psi^*(z) = \sum_n \psi^*_n z^{-n-1}.
\]

Putting $\om = :(T\psi)\psi^*:$ gives $\bigwedge$ a conformal structure of central charge $c = -2$, in which $\D(\psi) = 0$ and $\D(\psi^*)=1$. If we put $\al = :\psi \psi^*:$ then we have
\begin{align*}
[\al_\la \al] &= \la \\
\text{and} \quad [L_\la \al] &= (T+\la)\al - \la^2.
\end{align*}
Let us put
\[
\Theta(\tau, z | u) = \str_{\bigwedge} u_0 e^{-2\pi i z (\al_0-1/2)} q^{L_0-c/24},
\]
it is straightforward to see that $\Theta(\tau, z | \vac) = \Theta(\tau, z)$. It is known that $\bigwedge$ is $C_2$-cofinite and rational, and that the unique irreducible $\bigwedge$-module is $\bigwedge$. Hence Theorem \ref{theorem1.body} and (\ref{classical.theta}) imply
\[
\Theta\left(-1/\tau, z/\tau | \tau^{-L_{[0]}} e^{-\frac{z}{\tau} {I(\al)}} u\right) = -i e^{\pi i z^2/\tau} \Theta(\tau, z | u).
\]

\section{Principal Affine $W$-Algebras}\label{section.principal.affine}

The affine $W$-algebras form a large and interesting class of vertex algebras. From the data of a finite dimensional simple Lie (super)algebra $\ov\g$, nilpotent element $f \in \g$, and level $k$, the algebra $\W^k(\ov\g, f)$ is obtained via quantised Drinfeld-Sokolov reduction, i.e., as cohomology of the BRST complex, of the affine vertex algebra $V^k(\ov\g)$. See \cite{FeiginFrenkel} for $f$ principal nilpotent, and \cite{KRW03} for the general case.

In this section we study trace functions and their modular transformations for modules of $\W_k(\ov\g)$ the simple quotient of the universal affine $W$-algebra associated with principal nilpotent element $f$, and admissible number $k$. The article \cite{FKW} is an excellent reference.

\subsection{The BRST Complex}


Let $\ov{\g}$ be a finite dimensional simple Lie algebra as in Section \ref{liealgebras}. Let $\{e_\al\}_{\al \in \ov{\D}}$ be a root basis of $\ov{\g}$. For $\beta, \ga \in \ov{\D}_+$ the structure constants $c_{\beta, \ga}^\al$ are defined by $[e_\beta, e_\ga] = \sum c_{\beta, \ga}^\al e_\al$. For $\al \in \ov{\D}_+$ we denote by $\varphi_\al \in \ov{\n}_\pm$ the element of $\ov{\n}_\pm$ corresponding to $e_{\pm \al}$, and $\varphi_\al^* \in \ov{\n}_\pm^*$ its dual.


Let $\bigwedge_\pm$ be the Fock $\cliff(\ov{\n}_\pm)$-module generated from a vector $\vac$ with relations $\varphi_{\al, n\geq 0}\vac=0$, $\varphi^*_{\al, n\geq 1}\vac=0$ in the case of $\bigwedge_+$, and relations $\varphi_{\al, n\geq 1}\vac=0$, $\varphi^*_{\al, n\geq 0}\vac=0$ in the case of $\bigwedge_-$. Once again, as in {\cite[Section 3.6]{KacVA}}, $\bigwedge_\pm$ is a vertex algebra, with generating fields
\[
\varphi(z) = \sum_n \varphi_n z^{-n-1} \quad \text{and} \quad \varphi^*(z) = \sum_n \varphi^*_n z^{-n},
\]
in the case of $\bigwedge_+$, and
\[
\varphi(z) = \sum_n \varphi_n z^{-n} \quad \text{and} \quad \varphi^*(z) = \sum_n \varphi^*_n z^{-n-1},
\]
in the case of $\bigwedge_-$ (where $\varphi \in \ov{\n}_\pm$ and $\varphi^* \in \ov{\n}_\pm^*$). The vertex algebra $\bigwedge_\pm$ carries a $\Z$-grading induced by $\deg \vac = 0$, $\deg \varphi = -1$ and $\deg \varphi^* = +1$. It also carries a conformal structure
\[
\om^{\bigwedge, +} = \sum_{\al \in \ov{\D}_+} :(T\varphi_\al^*)\varphi_\al: \quad \text{or} \quad \om^{\bigwedge, -} = \sum_{\al \in \ov{\D}_+} :(T\varphi_\al)\varphi_\al^*:,
\]
which gives conformal weights $\D(\varphi)=1$, $\D(\varphi^*)=0$ in the case of $\bigwedge_+$, and $\D(\varphi)=0$, $\D(\varphi^*)=1$ in the case of $\bigwedge_-$.

Let $k \in \C$. The functor of quantised Drinfeld-Sokolov reduction on $V^k(\ov\g)$-modules (which comes in two variants: $+$ and $-$) is defined as follows. Let $M$ be a $V^k(\ov\g)$-module and put
\[
C^\bullet_\pm(M) = M \otimes \textstyle\bigwedge_\pm,
\]
which is a module over the vertex algebra $C^\bullet_\pm = C^\bullet_\pm(V^k(\ov\g))$. In $C_\pm^\bullet$ the element $Q_\pm = Q^{\text{st}}_\pm + p$ is defined by
\begin{align*}
Q^{\text{st}}_\pm &= \sum_{\al \in \ov{\D}_+} e_{\pm \al} \otimes \varphi^*_\al - \tfrac{1}{2} \sum_{\al, \beta, \ga \in \ov{\D}_+} c_{\beta\ga}^\al :\varphi_\al \varphi^*_\beta \varphi^*_\ga:, \quad \text{and} \quad
p = \sum_{\al \in \ov{\Pi}} \varphi^*_{\al}. 
\end{align*}
The module $C_\pm^\bullet(M)$ with $\Z$-grading induced from that on $\bigwedge_\pm$, is regarded as a complex, with the differential $d_\pm = (Q_\pm)_{(0)}$. The quantised Drinfeld-Sokolov reduction of $M$ is the cohomology $H^\bullet_\pm(M)$ of the complex $(C^\bullet_\pm(M), d_\pm)$.

In \cite{FKW} the following vectors of $C^\bullet_\pm$ were introduced.
\[
\widetilde{x} = x \otimes 1 + 1 \otimes F^x, \quad \text{where} \quad F^x = \sum_{\al \in \ov{\D}_+} \al(x) :\varphi_\al \varphi_\al^*:,
\]
for $x \in \ov{\h}$. The associated fields $\widetilde{x}(z)$ commute with $d_+^{\text{st}}$.

At noncritical level $k \neq -h^\vee$ the complex $C^\bullet_+$ carries a conformal structure  
\[
\om = \om^\text{Sug} \otimes 1 + 1 \otimes \om^{\bigwedge, +} + T \widetilde{\ov{\rho}^\vee},
\]
compatible with the differential. The central charge is
\[
c(k) = \ell - 12 \left[ (k+h^\vee)(\ov\rho^\vee, \ov\rho^\vee) - 2 (\ov\rho, \ov\rho^\vee) + \frac{(\ov\rho, \ov\rho)}{k+h^\vee} \right].
\]
Using the `strange formula' $h^\vee \dim{\ov{\g}} = 12(\ov\rho, \ov\rho)$ of Freudenthal and de Vries \cite{FdV.book}, we rewrite $c(k)$ as follows
\begin{align}\label{central.charge.conversion}
c(k) = c_k - 2|\ov{\D}_+| - 12\left[ (k+h^\vee)(\ov{\rho}^\vee, \ov{\rho}^\vee) - 2(\ov{\rho}, \ov{\rho}^\vee) \right].
\end{align}

We have the following $\la$-bracket relations:
\begin{align*}
[\widetilde{x}_\la \widetilde{x}'] &= (k + h^\vee) \la, \\
\text{and} \quad
[\om_\la \widetilde{x}] &= (T+\la) \widetilde{x} - \la^2 \left[ (k+h^\vee) (\ov{\rho}^\vee, x) - \ov{\rho}(x) \right],
\end{align*}
from \cite[Lemma 3.2]{FKW} and \cite[Theorem 2.4 (b)]{KRW03}, respectively.

For an arbitrary vertex algebra $V$ we put $\lie V = V[t, t^{-1}] / (T + \partial_t) V[t, t^{-1}]$, denoting by $u_{(m)}$ the image of $u t^m$. The formula (\ref{commutator.formula}) defines a Lie algebra structure on $\lie V$.

As explained in {\cite[Section 2.2]{FKW}}, there is a $C_+^\bullet$-module structure on $C_-^\bullet(M)$, implemented by a morphism
\[
\widetilde{w} : U(\lie C_+^\bullet) \rightarrow U(\lie C_-^\bullet).
\]
We just need the following formulas (which follow from (3.1.6), (2.2.4), and (2.2.6) of \cite{FKW}):
\[
\widetilde{w}(L_0) = L_0^\text{Sug} + L_0^{\wedge, -} + (\ov\rho, \ov\rho^\vee) - \frac{k+h^\vee}{2}(\ov\rho^\vee, \ov\rho^\vee),
\]
and
\[
\widetilde{w}(\widetilde{x}_0) = \widetilde{(\ov{w}^0 x)}_0 + (k+h^\vee) (\ov{\rho}^\vee, x),
\]
where $\ov{w}^0$ is the longest element of the finite Weyl group $\ov W$.

For $u \in C^\bullet_+$, and $x \in \ov\h$, put
\begin{align}\label{up.trace}
\Psi_\la(\tau, x | u) = \str_{C^\bullet_-(L(\la))} u_0 e^{2\pi i \left[\tilde{x}_0 - (k+h^\vee)(\ov{\rho}^\vee, x) + \ov{\rho}(x)\right]} q^{L_0 - c(k)/24}.
\end{align}
Note that these functions specialise (under $x = zh$) to the supertrace functions of Definition \ref{F.definition}. Substituting the formulas above for the $C^\bullet_+$-action on $C^\bullet_-(L(\la))$, and using (\ref{central.charge.conversion}), yields
\begin{align*}
\Psi_{\la}(\tau, x | u)
&= \str_{C_-^\bullet(L(\la))} \widetilde{w}(u_0) e^{2\pi i \left[ \widetilde{w}(\widetilde{x}_0) - (k+h^\vee)(\ov{\rho}^\vee, x) + \ov{\rho}(x) \right]} q^{\widetilde{w}(L_0) - c(k)/24} \\
&= q^{(\ov\rho, \ov\rho^\vee) - \frac{k+h^\vee}{2}(\ov\rho^\vee, \ov\rho^\vee) - c(k)/24} \str_{C_-^\bullet(L(\la))} \widetilde{w}(u_0) e^{2\pi i [\widetilde{\ov{w}^0(x_0)} + \ov{\rho}(x)]} q^{L_0^\text{Sug} + L_0^{\wedge, -}} \\
&= e^{2\pi i \ov{\rho}(x)} q^{(\ov\rho, \ov\rho^\vee) - \frac{k+h^\vee}{2}(\ov\rho^\vee, \ov\rho^\vee) - c(k)/24} \str_{C_-^\bullet(L(\la))} \widetilde{w}(u_0) e^{2\pi i \widetilde{\ov{w}^0(x_0)}} q^{L_0^\text{Sug} + L_0^{\wedge, -}} \\
&= e^{2\pi i \ov{\rho}(x)} q^{-(c_k - 2|\ov{\D}_+|)/24} \str_{C_-^\bullet(L(\la))} \widetilde{w}(u_0) e^{2\pi i \widetilde{\ov{w}^0(x_0)}} q^{L_0^\text{Sug} + L_0^{\wedge, -}}.
\end{align*}

Let us write
\[
\Theta_{\ov{\g}}(\tau, x) = \prod_{\al \in \ov{\D}_+} \Theta(\tau, \al(x)), \quad \text{for $\tau \in \HH$ and $x \in \ov\h$}
\]
where $\Theta$ is the theta function (\ref{theta.definition}). Then formula (\ref{classical.theta}), and the relation $\sum_{\al \in \ov{\D}_+} \al(x)^2 = h^\vee (x, x)$, implies
\begin{align}\label{affine.theta}
\Theta_{\ov{\g}}(-1/\tau, z/\tau) = (-i)^{|\D_+|} e^{\pi i h^\vee (x, x)/\tau} \Theta_{\ov{\g}}(\tau, x).
\end{align}

The commutation relations
\[
[F^x_0, \varphi_{\al, n}] = \al(x) \varphi_{\al, n} \quad \text{and} \quad 
[F^x_0, \varphi_{\al, n}^*] = -\al(x) \varphi_{\al, n}^*,
\]
in $\bigwedge_-$, together with the relation $\ov{\rho} = \tfrac{1}{2} \sum_{\al \in \ov{\D}_+} \al$, imply that
\begin{align}\label{supetrace.to.theta}
e^{2\pi i \ov{\rho}(x)} q^{|\D_+|/12} \str_{\bigwedge_-} q^{L^{\wedge, -}_0} e^{2\pi i F^x_0} = \Theta_{\ov{\g}}(\tau, x).
\end{align}

Let $k \in \Q$ be a principal \textup{(}resp.\ coprincipal\textup{)} number for $\ov\g$. The $V^k(\ov\g)$-module $L(\la)$ descends to a module over the simple quotient $V_k(\ov\g)$ if and only if $\la \in \prin^k$ \textup{(}resp.\ $\la\in \coprin^k$\textup{)}. Furthermore any $V_k(\ov\g)$-module from category $\OO_k$ is completely reducible.

\begin{thm}
Let $k \in \Q$ be a \textup{(co}\textup{)}principal admissible number for $\ov\g$, and $\la \in \coorprin^k$. Then the function $\Psi_\la$ of \textup{(}\ref{up.trace}\textup{)} satisfies
\begin{align}\label{up.transform}
\begin{split}
\Psi_\la\left( \frac{-1}{\tau}, \frac{x}{\tau} \bigg| \tau^{-L_{[0]}} \exp{\left[\frac{{I(x)}}{\tau}\right]} u \right)
= \exp{\left[ \pi i (k+h^\vee) \frac{(x, x)}{\tau} \right]}
\sum_{\mu \in \coorprin^k} (-i)^{|\ov{\D}_+|} a({\la, \mu}) \Psi_\mu(\tau, x | u)
\end{split}
\end{align}
where $a(\la, \mu)$ is the $S$-matrix of Theorem \ref{KWSmatrix} \textup{(}resp.\ Theorem \ref{KWSmatrix.coprin}\textup{)} for $k$ principal \textup{(}resp.\ coprincipal\textup{)}.
\end{thm}

\begin{proof}
We first show that the specialisation of (\ref{up.transform}) to $u = \vac$ holds. Indeed
\begin{align*}
\Psi_{\la}(\tau, x | \vac)
&= e^{2\pi i \ov{\rho}(x)} q^{-(c_k - 2|\ov{\D_+}|)/24} \tr_{L(\la)} q^{L_0} e^{2\pi i x_0} \cdot \str_{\bigwedge_-} q^{L_0^{\wedge, -}} e^{2\pi i F^x_0} \\
&= \chi_\la(\tau, x) \Theta_{\ov{\g}}(\tau, x),
\end{align*}
and the claim immediately follows from Proposition \ref{KWSmatrix} and equation (\ref{affine.theta}).

The vertex algebra $\bigwedge_+$ is $C_2$-cofinite and rational with unique irreducible module. Recall that $C_2$-cofiniteness implies cofiniteness relative to any splitting, and recall Lemma \ref{tensor.cofinite} on cofiniteness for tensor products. It follows that if $V_k(\ov\g)$ is rational and cofinite relative to $x \in \ov\h$, then
\[
C_+^\bullet = V_k(\ov\g) \otimes \textstyle{\bigwedge}_+
\]
is rational and cofinite relative to $\widetilde{x}$. The results of Section \ref{adm.aff.case} together with Theorem \ref{theorem1.body} now imply (\ref{up.transform}) for general $u$.
\end{proof}

\subsection{Trace functions of $W$-algebra Modules}

The \emph{\textup{(}principal\textup{)} affine $W$-algebra} associated with the simple Lie algebra $\ov\g$ at the level $k \in \C$ is the vertex algebra
\[
\W^k(\ov\g) = H^0(C^\bullet_+, d_+).
\]
The assignment
\[
M \mapsto H^0(C^\bullet_-(M), d_-)
\]
defines a functor $H^0_-(-)$ from $V^k(\ov\g)$-modules to $\W^k(\ov\g)$-modules.

Let $Z(\ov{\g})$ denote the centre of the universal enveloping algebra $U(\ov{\g})$. Each weight $\mu \in \ov{\h}^*$ yields a character $\ga_{\mu} : Z(\ov{\g}) \rightarrow \C$ via evaluation on the Verma $\ov\g$-module $M(\mu)$.

Starting with a conformal vertex algebra $(V, \om)$ Zhu constructed an associative algebra $\zhu(V)$ and an induction functor from $\zhu(V)$-modules to positive energy $V$-modules. This functor is a bijection on simple objects \cite{Z96}. The Zhu algebra of $\W^k(\ov\g)$ is isomorphic to $Z(\ov\g)$ {\cite[Theorem 4.16.3 (ii)]{A07.rep}} (see also {\cite[Proposition 3.3 (a)]{FKW}}). We denote by $\mathbb{L}(\ga)$ the $\W^k(\ov\g)$-module induced from the one dimensional $Z(\ov\g)$-module $\C \ga$ associated with the character $\ga$.

We now denote by $\W_k(\ov\g)$ the simple quotient of $\W^k(\ov\g)$. Let
\begin{align*}
\coorprin^k_{\text{nondeg}} &= \{ \la \in \coorprin^k | \text{ $\la(\al^\vee) \notin \Z$ for all $\al^\vee \in \ov{\D}^\vee$} \} \\
\text{and} \quad \coorprin^k_\W &= \{ \ga_{\ov{\la}} | \la \in \coorprin^k_{\text{nondeg}} \}.
\end{align*}
In the following theorem we summarise the results about $\W_k(\ov\g)$ and $H_-^0(-)$ that we shall use.
\begin{thm}\label{W.summary}
Let $\ov\g$ be a simple Lie algebra and $\g$ the associated untwisted affine Kac-Moody algebra.
\begin{enumerate}
\item {\cite[Theorem 7.6.1]{A07.rep}} For any level $k \in \C$ and module $M \in \OO_k$, one has $H^i_-(M) = 0$ for all $i \neq 0$.

\item {\cite[Corollary 7.6.4]{A07.rep}} For $k$ an admissible number, and $\la \in \coorprin^k_{\text{nondeg}}$, one has
\[
H^0_-(L(\la)) \cong \mathbb{L}(\ga_{-\ov{w}^0(\ov{\la})}),
\]
where $\ov{w}^0$ is the longest element in the finite Weyl group $\ov W$. If $\la \in \coorprin^k \backslash \coorprin^k_{\text{nondeg}}$, then
\[
H^0_-(L(\la)) = 0.
\]

\item {\cite[Theorem 10.4]{A12rational}} For $k$ a nondegenerate admissible number \textup{(}i.e., one for which $\coorprin^k_{\text{nondeg}}$ is nonempty\textup{)}, the set of irreducible $\W_k(\ov\g)$-modules is precisely
\[
\{ \mathbb{L}_\W(\ga) | \ga \in \coorprin_\W^k \}.
\]

\item For $k \in \Q$ a nondegenerate admissible number, the vertex algebra $\W_k(\ov\g)$ is rational {\cite[Theorem 10.10]{A12rational}} and $C_2$-cofinite {\cite[Theorem 5.10.2]{A.assoc.var}}.

\end{enumerate}
\end{thm}

Following the lead of \cite{FKW} we use the Euler-Poincar\'{e} principle to relate the character of $C^\bullet_-(M)$ to that of its cohomology $H^0_-(M)$, and thereby to compute the modular transformations of characters of the latter modules. The principle states that if $C^\bullet$ is a complex with finite dimensional components and $U$ is a degree $0$ endomorphism of $C^\bullet$ commuting with the differential then
\[
\str_{C^\bullet} U = \str_{H^\bullet(C^\bullet)} U.
\]

Now let $k$ be an admissible number, and $\la \in \coorprin^k$. Although the components of $C^\bullet_-(L(\la))$ are infinite dimensional, they are bigraded by $L_0$ and $\tilde{x}_0$ with finite dimensional pieces. It is thus valid to write
\[
\tr_{H^0_-(L(\la))} u_0 e^{2\pi i \tilde{x_0}} q^{L_0} = \str_{C^\bullet_-(L(\la))} u_0 e^{2\pi i \tilde{x_0}} q^{L_0}
\]
(for $u \in C^\bullet_+$ a chain). Let
\[
\psi_\la(\tau | u) = \tr_{H^0_-(L(\la))} u_0 q^{L_0-c(k)/24}.
\]
Then $\lim_{x \rightarrow 0} \Psi_\la(\tau, x | u) = \psi_\la(\tau | u)$ and passing to the $x \rightarrow 0$ limit of the relation (\ref{up.transform}) yields the following.
\begin{cor}\label{down.transform}
The functions $\psi_\la(\tau | u)$ satisfy
\[
\psi_\la(-1/\tau | \tau^{-L_{[0]}} u) = \sum_{\mu \in \coorprin^k} (-i)^{|\ov{\D}_+|} a({\la, \mu}) \psi_\mu(\tau | u),
\]
where $a(\la, \mu)$ is the $S$-matrix of Theorem \ref{KWSmatrix} \textup{(}resp.\ Theorem \ref{KWSmatrix.coprin}\textup{)} for $k$ principal \textup{(}resp.\ coprincipal\textup{)}.
\end{cor}
The special case $u = \vac$ of Corollary \ref{down.transform} was obtained as \cite[Proposition 4.4]{KW90}.

\subsection{Parametrisation of Irreducible {$\W_k(\overline{\g})$}-modules}


Let $k \in \Q$ be a principal admissible number for $\ov\g$. We define $p, q \in \Z$ by $k + h^\vee = p/q$, $q > 0$, and $(p, q)=1$. Let
\[
{I_{p, q} = \frac{P_{+}^{p-h^\vee} \times P_{+}^{\vee, q-h}}{\widetilde{W}_+},}
\]
where the action of $\widetilde{W}_+$ is $w(\la, \la') = (w\la, w\la')$. There is a bijection
\begin{align}\label{WI.bijection}
\begin{split}
\ov{W} \times I_{p, q} &\rightarrow \prin^k_{\text{nondeg}} \\
(\ov w, (\la, \la')) &\mapsto \ov w . \left( \ov{\la} - (k+h^\vee)(\ov{\la'}+\ov{\rho}^\vee) + (k+h^\vee)\La_0 \right).
\end{split}
\end{align}
This descends to a bijection
\begin{align}\label{I.bijection}
I_{p, q} \rightarrow \prin_\W^k.
\end{align}
For $(\la, \la') \in I_{p, q}$, we define $\mathbb{L}(\la, \la') = \mathbb{L}(\ga)$, where $\ga \in \prin^k_\W$ is the central character associated with $(\la, \la')$ via (\ref{I.bijection}). Let
\[
\varphi_{\la, \la'}(\tau | u) = \tr_{\mathbb{L}(\la, \la')} u_0 q^{L_0 - c(k)/24}.
\]
If $(\ov w, (\la', \la'')) \mapsto \la$ under the bijection above then $\varphi_{\la', \la''} = \psi_\la$.

The modular $S$-transformation of the functions $\varphi_{\la, \la'}(\tau | u)$ can be derived from Corollary \ref{down.transform}. For $u = \vac$ this derivation was carried out in \cite[Proposition 4.4]{KW90}. The same calculation yields the general result
\begin{cor}\label{S.matrix.factorised}
The trace functions $\varphi_{\la, \la'}(\tau | u)$ satisfy
\[
\varphi_{\la, \la'}(-1/\tau | (\tau)^{-L_{[0]}} u) = \sum_{(\mu, \mu') \in I_{p, q}} S_{(\la, \la'), (\mu, \mu')} \varphi_{\mu, \mu'}(\tau | u),
\]
where
\begin{align*}
S_{(\la, \la'), (\mu, \mu')}
= {} & (pq)^{-\ell/2} |J|^{-1/2} e^{2\pi i \left[ (\ov{\la}'+\ov{\rho}, \ov{\mu}+\ov{\rho}) + (\ov{\la}+\ov{\rho}, \ov{\mu}'+\ov{\rho}) \right]} \\
&\times \sum_{y \in \ov{W}} \eps(y) e^{-\frac{2\pi i p}{q}(\ov{\la}'+\ov{\rho}, y(\ov{\mu}'+\ov{\rho}))} \sum_{w \in \ov{W}} \eps(w) e^{-\frac{2\pi i q}{p}(\ov{\la}+\ov{\rho}, w(\ov{\mu}+\ov{\rho}))}.
\end{align*}
\end{cor}

\subsection{Fusion Rules}

There is a notion of `fusion' tensor product $\dot{\otimes}$ on the category of $V$-modules over a suitably well-behaved vertex algebra $V$ (defined at this level of generality by Huang and Lepowsky in {\cite{HLtensorIandII}}). Let $\irr(V)$ denote the set of isomorphism classes of irreducible $V$-modules. The Verlinde formula posits a relationship between the fusion rules $\mathcal{N}$ defined by
\[
A \dot\otimes B \cong \bigoplus_{C \in \irr(V)} \mathcal{N}_{A, B}^C C,
\]
for $A, B \in \irr(V)$, and the coefficients $S_{A, B}$ of the $S$-matrix of the trace functions on irreducible $V$-modules. Namely
\begin{align}\label{Verlinde.formula}
\mathcal{N}_{A, B}^C = \sum_{X \in \irr(V)} \frac{S_{A, L} S_{B, L} S_{L, C'}}{S_{V, L}}.
\end{align}
Here $M'$ denotes the adjoint module of $M$ {\cite[Section 5]{B86}}. Formula (\ref{Verlinde.formula}) was proved (for suitably regular vertex algebras) by Huang in \cite{Huang08Contemp}.

The irreducible modules of the simple affine vertex algebra $V_k(\ov\g)$ at nonnegative integer level $k \in \Z_+$ are indexed by $P_+^k$. The $S$-matrix in this case is known since \cite{KP84}, and the fusion rules $\mathcal{N}_{\la, \mu}^\nu$ may thereby be determined via the Verlinde formula. Substantial effort has been devoted to efficient calculation of these coefficients in the physics literature \cite[Chapter 16]{DMS.CFT}.

Huang's proof of the Verlinde formula also applies to $\W_k(\ov\g)$. The chief technical conditions on $\W_k(\ov\g)$ that need to be verified are supplied by {Theorem \ref{W.summary} part (4)}. Using the Verlinde formula and Corollary \ref{S.matrix.factorised} it is possible to express the fusion rules of $\W_k(\ov\g)$ in terms of the fusion rules $\mathcal{N}_{\la, \mu}^\nu$ above. This calculation was done in \cite{FKW} for a simply laced $\g$. We quote the answer.
\begin{thm}[{\cite[Theorem 4.3]{FKW}}]
Let $\ov\g$ be simply laced and let $k = p/q - h^\vee$ as above be a principal admissible number for $\ov\g$. Assume that $(q, |J|) = 1$. Choose the representatives $(\la, \la'), (\mu, \mu'), (\nu, \nu') \in I_{p, q}$ such that $\ov{\la'}, \ov{\mu'}, \ov{\nu'} \in \ov{Q}$. Then one has the following expression for the fusion rules between irreducible $\W_k(\ov\g)$-modules:
\[
\mathcal{N}_{(\la, \la'), (\mu, \mu')}^{(\nu, \nu')} = \mathcal{N}_{\la, \mu}^{\nu} \mathcal{N}_{\la', \mu'}^{\nu'}.
\]
\end{thm}


\def\cprime{$'$} \def\cprime{$'$}

\end{document}